\documentclass[11pt,twoside]{article} 
\usepackage{bbm}
\usepackage{amsmath,amssymb}
\usepackage{hyperref}
\hypersetup{
    colorlinks,
    citecolor=black,
    filecolor=black,
    linkcolor=black,
    urlcolor=black}

\makeatletter
\newcommand*{\mint}[1]{%
  \mint@l{#1}{}%
}
\newcommand*{\mint@l}[2]{%
  \@ifnextchar\limits{%
    \mint@l{#1}%
  }{%
    \@ifnextchar\nolimits{%
      \mint@l{#1}%
    }{%
      \@ifnextchar\displaylimits{%
        \mint@l{#1}%
      }{%
        \mint@s{#2}{#1}%
      }%
    }%
  }%
}
\newcommand*{\mint@s}[2]{%
  \@ifnextchar_{%
    \mint@sub{#1}{#2}%
  }{%
    \@ifnextchar^{%
      \mint@sup{#1}{#2}%
    }{%
      \mint@{#1}{#2}{}{}%
    }%
  }%
}
\def\mint@sub#1#2_#3{%
  \@ifnextchar^{%
    \mint@sub@sup{#1}{#2}{#3}%
  }{%
    \mint@{#1}{#2}{#3}{}%
  }%
}
\def\mint@sup#1#2^#3{%
  \@ifnextchar_{%
    \mint@sup@sub{#1}{#2}{#3}%
  }{%
    \mint@{#1}{#2}{}{#3}%
  }%
}
\def\mint@sub@sup#1#2#3^#4{%
  \mint@{#1}{#2}{#3}{#4}%
}
\def\mint@sup@sub#1#2#3_#4{%
  \mint@{#1}{#2}{#4}{#3}%
}
\newcommand*{\mint@}[4]{%
  \mathop{}%
  \mkern-\thinmuskip
  \mathchoice{%
    \mint@@{#1}{#2}{#3}{#4}%
        \displaystyle\textstyle\scriptstyle
  }{%
    \mint@@{#1}{#2}{#3}{#4}%
        \textstyle\scriptstyle\scriptstyle
  }{%
    \mint@@{#1}{#2}{#3}{#4}%
        \scriptstyle\scriptscriptstyle\scriptscriptstyle
  }{%
    \mint@@{#1}{#2}{#3}{#4}%
        \scriptscriptstyle\scriptscriptstyle\scriptscriptstyle
  }%
  \mkern-\thinmuskip
  \int#1%
  \ifx\\#3\\\else_{#3}\fi
  \ifx\\#4\\\else^{#4}\fi
}
\newcommand*{\mint@@}[7]{%
  \begingroup
    \sbox0{$#5\int\m@th$}%
    \sbox2{$#5\int_{}\m@th$}%
    \dimen2=\wd0 %
    \let\mint@limits=#1\relax
    \ifx\mint@limits\relax
      \sbox4{$#5\int_{\kern1sp}^{\kern1sp}\m@th$}%
      \ifdim\wd4>\wd2 %
        \let\mint@limits=\nolimits
      \else
        \let\mint@limits=\limits
      \fi
    \fi
    \ifx\mint@limits\displaylimits
      \ifx#5\displaystyle
        \let\mint@limits=\limits
      \fi
    \fi
    \ifx\mint@limits\limits
      \sbox0{$#7#3\m@th$}%
      \sbox2{$#7#4\m@th$}%
      \ifdim\wd0>\dimen2 %
        \dimen2=\wd0 %
      \fi
      \ifdim\wd2>\dimen2 %
        \dimen2=\wd2 %
      \fi
    \fi
    \rlap{%
      $#5%
        \vcenter{%
          \hbox to\dimen2{%
            \hss
            $#6{#2}\m@th$%
            \hss
          }%
        }%
      $%
    }%
  \endgroup
}
\usepackage{mathrsfs}
\usepackage{amssymb}
\usepackage{amsmath}
\usepackage{amsthm}
\usepackage{amsfonts}
\usepackage{soul}
\usepackage{color}

\definecolor{myhighlightcolor}{rgb}{1,1,0}
\usepackage{graphicx}
\usepackage[active]{srcltx}
\usepackage{tikz}
\usepackage{pgflibraryarrows}
\usepackage{pgflibrarysnakes}

\usepackage[cp1252]{inputenc}

\usepackage{mathrsfs}
\usepackage{graphicx}

\usepackage[active]{srcltx}

\allowdisplaybreaks

\usepackage{titletoc}
\titlecontents{section}[0pt]{\addvspace{2pt}\filright}
              {\contentspush{\thecontentslabel\ }}
              {}{\titlerule*[8pt]{.}\contentspage}

\def\rr{{\mathbb R}}
\def\rn{{{\rr}^n}}

\def\nn{{\mathbb N}}

\def\fz{\infty}
\def\az{\alpha}

\def\loc{{\mathop\mathrm{\,loc\,}}}

\def\lz{\lambda}
\def\dz{\delta}
\def\bdz{\Delta}
\def\ez{\epsilon}

\def\bz{\beta}

\def\gz{{\gamma}}

\def\bint{{\ifinner\rlap{\bf\kern.35em--}
\int\else\rlap{\bf\kern.45em--}\int\fi}\ignorespaces}

\def\bbint{{\ifinner\rlap{\bf\kern.35em--}
\hspace{0.078cm}\int\else\rlap{\bf\kern.45em--}\int\fi}\ignorespaces}

\def\r{\right}
\def\lf{\left}

\newtheorem{thm}{Theorem}[section]
\newtheorem{lem}[thm]{Lemma}
\newtheorem{prop}[thm]{Proposition}
\newtheorem{rem}[thm]{Remark}
\newtheorem{defn}[thm]{Definition}
\numberwithin{equation}{section}

\title
{\Large\bf
A nonlinear Calder\'on-Zygmund $ L^2$-theory
for  the Dirichlet problem involving $ -|Du|^{\gamma}\Delta^N_p u=f$
{\footnotetext{
{\it  Mathematics Subject Classification}:35J25, 35J60, 35B65\endgraf
Q. Miao was supported by National Natural Science Foundation of China (Nos.12371199,12171031).
 F. Peng was supported by the Fundamental Research Funds for the
Central Universities.
Y. Zhou was supported by NSFC (No.12025102\&12431006) and by the Fundamental Research Funds for the Central Universities.
\endgraf}
}
}

\author{Qianyun Miao, Fa Peng and Yuan Zhou}
\begin{document}

\arraycolsep=1pt
\allowdisplaybreaks
 \maketitle

\begin{center}
\begin{minipage}{13.5cm}\small
 \noindent{\bf Abstract.}\quad
We establish a nonlinear Calder\'on-Zygmund $L^2$-theory
 to the Dirichlet problem
$$-|Du|^{\gamma}\Delta^N_p u=f\in L^2(\Omega)\quad {\rm in}\quad \Omega; \quad u=0 \ \mbox{on $\partial\Omega$} $$
for  $n\ge2$, $ p>1$ and  a  large range of  $\gamma>-1$,
 in particular, for all $p>1$ and  all $ \gamma>-1$ when $n=2$.
Here
  $\Omega\subset \mathbb{R}^n$   is a bounded convex domain, or a
  bounded  Lipschitz domain
 whose boundary has small  weak second fundamental form in the sense of Cianchi-Maz'ya (2018).

 \medskip
 \noindent \quad\
The proof relies on an extension of  an Miranda-Talenti \& Cianchi-Maz'ya type inequality,
that is,
for any $v\in C^\infty_0(\Omega)$ in any bounded smooth domain $\Omega$,
$\|D[(|Dv|^2+\epsilon)^{\frac\gamma 2}Dv]\|_{L^2(\Omega)}$ is bounded via
$\|(|Dv|^2+\epsilon)^{\frac\gamma 2} \Delta^N_{p,\epsilon}v  \|_{L^2(\Omega)}$,
where   $\Delta^N_{p,\epsilon}v$ is the $\epsilon$-regularization of  normalized $p$-Laplacian.

 \medskip
 \noindent \quad\
Our results extend the well-known Calder\'on-Zygmund $L^2$-estimate for the Poisson equation,
  a nonlinear global second order Sobolev estimate for inhomogeneous $p$-Laplace equation by  Cianchi-Maz'ya (2018),
 and a local  $W^{2,2}$-estimate for inhomogeneous normalized $p$-Laplace equation by Attouchi-Ruosteenoja (2018).

\end{minipage}
\end{center}

\tableofcontents
\section{Introduction}

For $ 1<p<\fz$,  denote by $\Delta_p$  the $p$-Laplacian  and by $\Delta^N_p$ its normalization, that is,
  $$\Delta_pu:={\rm div}(|Du|^{p-2}Du)\quad \mbox{and}\quad \Delta^N_pu:=|Du|^{2-p}\Delta_pu.
  $$
 In the case $p=2$,    both of them are the Laplacian $\Delta$, that is,
 $ \Delta u:={\rm div}(Du).$
For $\gz>-1$, let us consider the Dirichlet problem
  \begin{align}\label{plapgz}
-|Du|^\gamma \Delta^N_p u =f\quad {\rm in}\quad \Omega;\quad u=0 \quad\mbox{on $\partial \Omega$.}
\end{align}
If  $ f\in C^0(\Omega)$, we refer to Section 2 for the definition of viscosity solutions to the equation
$-|Du|^\gamma \Delta^N_p u =f$.
Under  $ f\in C^0(\Omega)\cap L^\infty(\Omega)$, the existence of viscosity solution $u\in C^0(\Omega)$ to \eqref{plapgz} was already known;
see Birindelli-Demengel \cite{bd09} and also Section 6.

 In this paper we deal with
  $f\in L^2(\Omega)$ and approximation solutions to the problem \eqref{plapgz}.
A function $u:\Omega\to\mathbb R$  is called as an approximation solution to the problem
\eqref{plapgz} with $f\in L^2(\Omega)$ if
  $u^\ez\to u$ in   $W^{1,q}_0(\Omega)$ for some $q>1$ as $\ez\to0$ (up to some subsequence),
and  $u^\ez$ is a  viscosity solution to the approximation problem
 \begin{equation}\label{ez-app0}
-|Dv|^\gamma \Delta^N_p v=f^\epsilon \ \mbox{in $\Omega$};\quad  v=0 \ \mbox{on $\partial\Omega$},
\end{equation}
where
 $f^{\ez}\in L^\infty(\Omega)\cap C^0(\Omega)$
 and $f^\ez\to f$ in $L^2(\Omega)$ as $ \ez\to0$.

The main purpose of this paper is to establish  the following
   nonlinear Calder\'on-Zygmund $L^2$-theory for the problem \eqref{plapgz},
   including the  existence of approximation (viscosity) solutions to
the problem \eqref{plapgz}, a global first order quantitative  Sobolev regularity,
 and  a   nonlinear global second order Sobolev  regularity.
 A triplet $ (n,p,\gamma)$ is admissible
  if $n\ge2$, $ p>1$ and $ \gz>\gamma_{n,p}$, where
     \begin{align}\label{pgz}
    \gz_{n,p}:=-1+(p-1)\frac{n-2}{2(n-1)},\quad \mbox{in particular,}\quad \gamma_{2,p}:=-1.
   \end{align}
 Moreover, if $\gz=0$, then the triplet $ (n,p,0)$ is admissible if and only
 if $1<p<3+\frac{2}{n-2}$.
   Write $2^\star=\frac{2n}{n-2}$ when $n\ge3$ and $2^\star=q$ for any sufficiently large $q<\infty$.
\begin{thm}\label{th-re1}
Suppose that the triplet $ (n,p,\gamma)$ is admissible
and  $\gamma>-\frac{4}{n+2}$.
 Let  $\Omega\subset\rn$ be  a bounded convex domain.
 For any  $f\in  L^2(\Omega)$,
  the problem  \eqref{plapgz} admits an approximation solution
   $ u\in W^{1,2}_0(\Omega)$, which satisfies
      \begin{equation}\label{thxx} \mbox{
$Du\in L^{(1+\gz)2^\star}(\Omega)$ and
$|Du|^{\gz}Du\in W^{1,2}(\Omega)$}
\end{equation}
and has the upper bound
\begin{align}\label{th1-1}
\|Du\|^{\gz+1}_{L^{(1+\gz)2^\star}(\Omega)}+
\|D[|Du|^{\gz}Du]\|_{L^2(\Omega)}\le C\|f \|_{L^2(\Omega)}
\end{align}
for some constant $C=C(p,n,\gz)$.
If  $f\in C^0(\Omega)$ in addition, then $u\in C^0(\Omega)$ is indeed a viscosity solution to  \eqref{plapgz}.
  \end{thm}

  Next we relax the convexity of domains required in Theorem \ref{th-re1} to bounded Lipschitz domains with small weak second fundamental form.
Given a bounded Lipschitz domain $ \Omega$, denote by $d_\Omega$ the diameter of $ \Omega$ and by $L_\Omega$  the Lipschitz constant of $\Omega$.
Assume  that   the boundary   $\partial\Omega\in W^2L^{n-1,\infty}$ when $n\ge 3$ and
 $\partial\Omega\in W^2L^{1,\infty}\log L$ when $n=2$.  Then
 $\partial\Omega$  enjoys   a weak second fundamental form $\mathcal{B}$ in $L^{n-1,\infty}(\partial\Omega)$ when $n\ge 3$ and
$L^{1,\infty}\log L(\partial\Omega)$ when $n=2$.
For $r>0$, define   \begin{equation*}
   \Psi_{\Omega}  (r)  :=\left\{\begin{array}{lc}\displaystyle
  \sup_{x\in\partial\Omega}\|\mathcal{B}\|_{L^{n-1,\infty}(\partial\Omega\cap B_r(x))}  \quad&\textrm{if}\quad n\geq 3,\\\displaystyle
   \sup_{x\in\partial\Omega}\|\mathcal{B}\|_{L^{1,\infty}\log L (\partial\Omega\cap B_{r}(x))} \quad&\textrm{if}\quad n=2.
\end{array}\right.
\end{equation*}
For more details see \cite{cm18}.
Below, when writing  $\Psi_\Omega$   we always make above assumption on $\partial\Omega$ without further specifying.


\begin{thm}\label{th-re1a}
Suppose that the triplet $ (n,p,\gamma)$ is admissible
and  $ -\frac{4}{n+2}< \gz\le p-2$.
 Let  $\Omega\subset\rn$ be  a bounded Lipchitz domain but not convex.

There exists   small positive constant $\dz=\dz(p,n,\gz,
L_{\Omega}, d_{\Omega})$ such that if
$\lim_{r\to 0^+}\Psi_{\Omega}(r)<\dz$, then
for any  $f\in  L^2(\Omega)$,
  the problem  \eqref{plapgz} admits an approximation solution
   $ u\in W^{1,2}_0(\Omega)$,
  which satisfies   \eqref{thxx} and \eqref{th1-1} with the constant $C=C(p,n,\gz,
L_{\Omega}, d_{\Omega})$.
  If  $f\in C^0(\Omega)$ in addition,  $u\in C^0(\Omega)$ is indeed a viscosity solution to  \eqref{plapgz}.
  \end{thm}

In the case $p=2$ and  $ \gamma>  -\min\{\frac{4}{n+2},\frac{n}{2n-2}\}$ but $\gz\ne0$,
Theorem \ref{th-re1}  actually gives a nonlinear  Calder\'on-Zygmund  $L^2$-theory to the Dirichlet zero boundary problem
for the equation   $|Du|^\gamma\Delta u=f\in L^2$, which is completely new.
If $\Omega$ is not convex, in the range $0> \gamma>  -\min\{\frac{4}{n+2},\frac{n}{2n-2}\}$, Theorem \ref{th-re1a} also gives a similar result.

  In the case $\gamma=0$ and $ 1<p<3+\frac2{n-2}$,  applying Theorem \ref{th-re1}  to   the Dirichlet zero boundary problem for the equation $ \Delta^N_p u =f$,
we  obtain a global
 $L^2$-estimate of Hessian $D^2u$, that is, $ \|D^2u\|_{L^2(\Omega)}\le C\|f\|_{L^2(\Omega)}$, which is also completely new. Also, in the range $ 2<p<3+\frac2{n-2}$, Theorem \ref{th-re1a}  shows a similar result when $\Omega$ is not convex.

The proofs  of Theorem \ref{th-re1} and Theorem \ref{th-re1a} rely on the following crucial
  a priori estimates, which has its own interests.
 For $ 1<p<\fz$ and $\ez\in(0,1)$,  denote by   $\Delta_{p,\ez}$
   the
    $\ez$-regularization of  $\Delta_p$, and by $\Delta^N_{p,\ez}$ the $\ez$-regularization
    of $\Delta^N_p$, that is,
    $$ \ \Delta_{p,\ez}u:={\rm div}((|Du|^2+\ez)^{\frac{p-2}2}Du), \quad
    \Delta^N_{p,\ez}u:=(|Du|^2+\ez)^{\frac{2-p}2}\Delta_{p,\ez}u.$$

\begin{thm}\label{thm-v}
 Suppose that  $ (n,p,\gamma)$ is admissible.
    Let  $ \Omega\subset\rn$ be a bounded smooth domain, and let
$v\in C^\fz(\overline \Omega)$ with $v=0$ on $\partial\Omega$.
For any $\ez\in (0,1)$, the following holds.

\begin{enumerate}
 \item [(i)] If $\Omega$ is convex, then
 \begin{align}\label{thx1-1}
&\|Dv\|^{\gz+1}_{L^{(1+\gz)2^\star}(\Omega)}+
\|D[(|Dv|^2+\ez)^{\frac\gz 2}Dv]\|_{L^2(\Omega)}\nonumber\\
&\quad\quad\le C\|(|Dv|^2+\ez)^{\frac\gz 2} \bdz^N_{p,\ez}v  \|_{L^2(\Omega)}+C\ez^{\frac {\gz+1}2},
\end{align}
where $C=C(p,n,\gz)$.

  \item [(ii)] Assume that $\Omega$ is not convex and that
$\lim_{r\to 0^+}\Psi_{\Omega}(r)<\dz$ for some sufficiently small positive constant $\dz=\dz(p,n,\gz,
L_{\Omega}, d_{\Omega})$.
\begin{itemize}
\item[$\bullet$]
If $\gz\le p-2$, then \eqref{thx1-1} holds with $C=C(p,n,\gz,
L_{\Omega}, d_{\Omega})$.

\item[$\bullet$]
If $\gz>p-2$, then
\begin{align}\label{th1-2}
&\|Dv\|^{\gz+1}_{L^{(1+\gz)2^\star}(\Omega)}+
\|D[(|Dv|^2+\ez)^{\frac\gz 2}Dv]\|_{L^2(\Omega)}\nonumber\\
&\quad \le C\|(|Dv|^2+\ez)^{\frac\gz 2} \bdz^N_{p,\ez}v  \|_{L^2(\Omega)}+ C  \|v\|^{\gz+1}_{L^2(\Omega)}+C\ez^{\frac {\gz+1}2}.
\end{align}
with $C=C(p,n,\gz,
L_{\Omega}, d_{\Omega})$.

\end{itemize}

\end{enumerate}

\end{thm}

Below are several remarks.
\begin{rem} \rm

(i) If $f\in L^2(\Omega)$ but $ f\notin C^0(\Omega)$, we do not know whether the approximation solutions in Theorems \ref{th-re1} and
 \ref{th-re1a} solve the problem \eqref{plapgz} in some sense.

 (ii)  The uniqueness of approximation (viscosity) solutions to
 the problem \eqref{plapgz} is unavailable in general.
 Indeed, under $f\in C^0(\Omega)\cap L^\infty(\Omega)$,
 if $f>0$ in $\Omega$ or $f<0$ in $\Omega$, then the problem \eqref{plapgz} with $ \gamma=0$ has unique solutions;
 if $f$ changes sign, it remains open for the uniqueness of viscosity solutions to
 the problem \eqref{plapgz} with $ \gamma=0$;
 see \cite{m23} for details.

(iii) Theorems \ref{th-re1a}  does not include the case $\gz>p-2$
 since it is unclear for us whether  $ \|u\|_{L^2(\Omega)}^{1+\gz}$ is finite a priori under $f\in L^2(\Omega)$ or even $f\in C^0(\Omega)\cap L^2(\Omega)$.
Recall that $ \|u\|_{L^2(\Omega)}^{1+\gz}$ appears in \eqref{th1-2} of Theorem \ref{thm-v}. For further discussion, see  Remark \ref{rem7.4}.

 However, under stronger assumption $f\in L^\infty(\Omega)$,
 the case $ \gz>p-2$  is allowed;  see Theorem \ref{th-re0} (ii) below.

(iv)  Note that  $\gz>\frac{n-4}2$ if and only if  $2^{\star}(\gz+1)>n$, which  allows to use
 Morrey-Sobolev embedding inequality to conclude that  approximation (viscosity) solutions
 in Theorems \ref{th-re1} and \ref{th-re1a} belongs to the global H\"older space $C^\alpha (\overline\Omega)$.

  When $ n=2$, we always have  $2^{\star}(\gz+1)>2.$    Since $\frac4{n+2}=1$,
the condition $ \gamma>-1$ implies $\gamma>-\frac4{n+2}$ required  in Theorem \ref{th-re1} .

  When    $n\ge3$,     $(1+\gz)2^\star\le n$ if and only if   $\gz\le \frac{n-4}2$, and $-\frac{4}{n+2}<\gz$ if and only if   $(1+\gz)2^\star>2$.
  This  allows to get the existence of approximation (viscosity) solution $u\in W^{1,2}_0(\Omega)$ in Theorem \ref{th-re1}.
    However, in the case  $n\ge3$,   if $-\frac4{n+2}\ge \gz>\gamma_{np}
    $, it is unclear for us
 whether we can remove
the additional restriction  $\gamma>-\frac4{n+2}$
in Theorems \ref{th-re0} and  \ref{th-re1}
or not.
This restriction comes only
from  Lemma \ref{con-gr} (ii) and it is  used to guarantee some Sobolev convergence of solutions to approximation equations.
 See Remark  \ref{zzzz1} for more details and  possible ways to remove the requirement $\gamma>-\frac4{n+2}$ in Lemma \ref{con-gr} (ii).

 \end{rem}

 \subsection{ Motivation and related studies in the literature}

 Theorem \ref{thm-v} is  motivated by the Miranda-Talenti formula (see \cite{t65})
  and a regular $p$-Laplace version  by Cianchi-Maz'ya (see \cite{cm18}).
  Miranda-Talenti formula  leads to  a Calder\'on-Zygmund $L^2$-estimate of the  Dirichlet problem for
   $  \Delta  u=f$, while
  Cianchi-Maz'ya (see \cite{cm18}) build up a nonlinear $L^2$-estimates for   the  Dirichlet problem of the equation
  $\Delta_pu =f$.
Moreover, Attouchi and Ruosteenoja \cite{ar18} obtained local $W^{2,2}_\loc$ regularity for viscosity solutions to
    \eqref{plapgz} under $ f\in C^0(\Omega)\cap W^{1,1}(\Omega)$ for some special $p$ and $\gz$ based on the Cordes approach (see \cite{mps00}) and the Miranda-Talenti formula.
 Theorems \ref{th-re1} and \ref{th-re1a}   actually extend or improve such second order regularity; for details see below.    Recall  from  \cite{bd10,apr17, ar18} that
  viscosity solutions to $|Du|^\gz \Delta_p^Nu=f\in C^0(\Omega)$
always have local $C^{1,\beta}$ regularity.

The  famous Miranda-Talenti inequality states that
 \begin{align}\label{poss1}
 \|D^2u\|_{L^2(\Omega)}\le \|\Delta u\|_{L^2(\Omega)}
\quad\forall u\in W^{1,2}_0(\Omega) \cap W^{2,2}(\Omega),
\end{align}
where $ \Omega$ is a bounded convex domain with $C^2$-boundary; see \cite{t65}.
The Miranda-Talenti inequality   leads to the Calder\'on-Zygmund $L^2$-estimate for the Possion equation, that is, given any $ f\in L^2(\Omega)$,   the  Dirichlet problem
 $$-\bdz u =f\quad \mbox{in}\ \Omega;\quad u=0\quad \mbox{on}
  \partial \Omega $$
admits a unique solution $ u\in W^{1,2}_0(\Omega)\cap W^{2,2}(\Omega)$
with
  \begin{align}\label{poss}
 \int_{\Omega}|D^2u|^2\,dx\le C\int_{\Omega}f^2\,dx.
 \end{align}This was first established by Bernstein   \cite{be04} when $n=2$ and by Schauder  \cite{sc34} when $n\ge 3$.
There has been efforts to extend the class of domains so that  \eqref{poss} holds.
The Miranda-Talenti inequality and the Calder\'on-Zygmund $L^2$-estimate
for Poisson equations were further generalized  to bounded Lipschitz domains,
which
 are  either convex domains or have smallness regularity on $\Omega$, but do not necessarily have $ C^2$-boundary; see \cite{cm18}.

Miranda-Talenti inequality was employed by Cordes \cite{mps00} to establish the  Calder\'on-Zygmund $L^2$-estimate  for   Dirichlet problem of   second order linear elliptic equation
 \begin{align}\label{LA}
L_A(u): ={\rm tr}\{ AD^2u\} =f\quad \mbox{in} \ \Omega; u=0\ \mbox{on} \  \partial \Omega,
  \end{align}
where ${\rm tr}(B)$ is the trace of $B$ with any $n\times n$ matrix  $B$.
 Here  the coefficient  $A:\Omega\to\mathbb R^{n\times n}
   $ is symmetric, measurable  satisfying the elliptic condition \begin{align}\label{elliptic}
   \lz |\xi|^2\le \langle    A\xi,\xi\rangle\le \Lambda |\xi|^2\quad\forall \xi\in\rn \mbox{ for some}\   0<\lz\le  \Lambda<\fz.
   \end{align}
   To be precise,
  if   $A$ further satisfies the Cordes condition (see \cite{mps00})
 \begin{align}\label{cordes}\sum_{i,j=1}^na_{ij}^2\le \frac1{n-1+\delta}\left(\sum_{i=1}^na_{ii}\right)^2\mbox{for some}\ \dz>0,
 \end{align}
 then, thanks to \eqref{poss1},    the Dirichlet  problem \eqref{LA}
 admits a unique (strong) solution $ u\in W^{1,2}_0(\Omega)\cap W^{2,2}(\Omega)$ satisfying
     \begin{align}\label{possLA}
     \|D^2u\|_{L^2(\Omega)}\le  \frac1{1-\sqrt{1-\dz}} \sup_{\Omega}\frac{{\rm tr}(A)  } {|A|^2}\|f\|_{L^2(\Omega)}.
\end{align}

In dimension $n=2,$ the elliptic condition always implies  the Cordes condition.
In dimension $n\ge 3$,  the Cordes condition  does not necessarily follow from the elliptic condition, and moreover,  it is sharp to get the solvability in $ W^{1,2}_0(\Omega)\cap W^{2,2}(\Omega)$ of  this problem.  For
 more details
  we refer to
  \cite{mps00} and the references therein.

 The above Calderon-Zygmund $L^2$-estimate   and approach  for the problem \eqref{LA} was adapted to study   the Dirichlet problem:
 $$ -\Delta_p^Nu=f \quad\mbox{in}\ \Omega;u=0\quad\mbox{on}\ \partial \Omega,$$
  where  $1<p<3+\frac2{n-2} $.
 Such restriction comes from the Cordes condition as observed by Manfredi-Weitsman \cite{mw88}; indeed,
  if $ Du\ne 0$, writing
  $$-\Delta_p^Nu:=-\bdz u-(p-2)\frac{\langle D^2uDu,Du\rangle}{|Du|^2} =-{\rm tr}\left\{\left[ I_n+\frac{Du\otimes Du}{|Du|^2}\right] D^2u\right\},
  $$
we see that the coefficient
$$
\left[ I_n+(p-2)\frac{Du\otimes Du}{|Du|^2}\right]
\ \mbox{ satisfies \eqref{cordes}
if and only if  }\  1<p<3+\frac2{n-2}.$$

Via  a local version of the Cordes condition,  if $1<p<3+2/(n-2)$ and $\gz= 0$,
 Attouchi and Ruosteenoja \cite{ar18}  demonstrated a  local $W^{2,2}$-regularity of
 viscosity solutions to  inhomogeneous normalized $p$-Laplace equation $ \Delta_p^Nu=f\in C^0(\Omega)$ in $ \Omega$.
  Moreover,  when $p$ close to 2, for some small
  $\gz$, they also obtained the local
  $W^{2,2}$ estimate for equation \eqref{plapgz} with $f\in W^{1,1}(\Omega)\cap C^0(\Omega)$.

Recently, Cianchi-Maz'ya established a global second order regularity  for the Dirichlet problem involving   $p$-Laplacian $\Delta_p$
 with   $1<p<\infty$ in the seminal paper \cite{cm18}:
\begin{align}\label{plap}
-\Delta_pu:=-{\rm div}(|Du|^{p-2}Du)=f\quad {\rm in}\quad \Omega;\quad u=0 \quad\mbox{on $ \partial \Omega$.}
\end{align}
  To be precise,   for any
 $ f\in L^2(\Omega)$, this
 problem admits a unique (generalized) weak solution $u$ with
its nonlinear gradient $ |Du|^{p-2}Du\in W^{1,2}(\Omega)$ and
 \begin{align}\label{plap-es}
  \||Du|^{p-2}Du\|_{W^{1,2}(\Omega)}\le C\|f\|_{L^2(\Omega)},
\end{align}
where  $\Omega$ is required to be  convex or  $\Psi_\Omega$ is required to be small.
 See  \cite{cm19,bcdm22, accfm} for further extension to quasilinear equations and  system.
 A core ingredient to obtain \eqref{plap-es}  is
 a nonlinear version   of Miranda-Talenti inequality  by   Cianchi-Maz'ya \cite{cm18}, that is,
  \begin{align}\label{insd-eseqcm}
  \||D^2v|(|Dv|^2+\ez)^{ \frac{p-2}2}\|_{L^2(\Omega)}\le C\|\Delta_{p,\ez}v\|_{L^2(\Omega)}
  +C\|(|D v |^2+\ez)^{\frac{p-1}2}\|_{L^1(\Omega)}
\end{align}

 Our a priori estimates given in Theorem \ref{thm-v}  extend  the Miranda-Talenti formula \eqref{poss1}
 and a regular $p$-Laplace version \eqref{insd-eseqcm} by Cianchi-Maz'ya \cite{cm18}.
 Theorem \ref{thm-v}  allows to get a
 nonlinear Calder\'on-Zygmund $L^2$-estimates for the  Dirichlet problem  \eqref{plapgz} with suitable $p,\gz$ in Theorems \ref{th-re1}-\ref{th-re1a}.
   This extends the classical  Calder\'on-Zygmund $L^2$-estimates
   for Poisson equation and also the $p$-Laplace  version \eqref{plap-es} by Cianchi-Maz'ya
   \cite{cm18}.
In particular, applying Theorems \ref{th-re1}-\ref{th-re1a} to the normalized $p$-Laplace equation $\Delta^N_pu=f$,
under $ f\in L^\fz(\Omega)\cap C^0(\Omega)$ we improve the above results by   Attouchi and Ruosteenoja \cite{ar18} from local  $L^2$ bound to global $L^2$ bound of Hessian.
Applying Theorems \ref{th-re1}-\ref{th-re1a} to $ |Du|^\gz\Delta u=f$, we  get the global second order estimates.
This contributes the  regularity study for the equation $ |Du|^\gamma\Delta_p^N u=f$, in particular,  $|Du|^\gamma \Delta u=f
$ and also $\Delta^N_pu=f$.

 \subsection { Sketch the ideas to prove  Theorems \ref{th-re1}-\ref{thm-v}}

We begin with the proof of  Theorem \ref{thm-v}. Firstly we establish  an algebraic inequality in
Lemma \ref{key-in}, which is   based on a fundamental inequality (see Lemma
\ref{fun-in}) and a differential inequality
by Cianchi-Maz'ya \cite[Lemma 3.1]{cm18}.  Next applying
Lemma \ref{key-in}, we are able to
 bound  $ \|(|Dv|^2+\ez)^{\frac{\gz+\beta} 2}|D^2v|\|_{L^2(\Omega)}$
 via  $ \|(\bdz^N_{p,\ez}v)(|Dv|^2+\ez)^{\frac{\gz+\beta} 2}Dv\|_{L^{2}(\Omega)}$ when
 $\Omega$ is convex, or when  $\Omega$ is not convex and $\gz<p-2$. When
 $\Omega$ is not convex and $\gz>p-2$, then $ \|(|Dv|^2+\ez)^{\frac{\gz+\beta} 2}|D^2v|\|_{L^2(\Omega)}$ are bounded by $ \|(\bdz^N_{p,\ez}v)(|Dv|^2+\ez)^{\frac{\gz+\beta} 2}Dv\|_{L^{2}(\Omega)}$ and the additional term
  $ \|(|Dv|^2+\ez)^{\frac{\gz+\beta+1} 2}Dv\|_{L^{2}(\Omega)}$; see Lemma \ref{sd-es}.
 To get Theorem \ref{thm-v},  thanks to   Lemma \ref{sd-es} with $\beta=0$,  it remains to  bound  $ \|Dv\|_{L^{2^\star(1+\gamma)}(\Omega)}.$
This will be done in Lemma \ref{Lq-es}. Indeed, if $\gz<p-2$, we bound
  $\|Dv\|_{L^{2^{\star}(\gz+1)}(\Omega)}$ by
 $\|(|Dv|^2+\ez)^{\gz/2}\bdz^N_{p,\ez}v\|_{L^2(\Omega)}$. However,
if $\gz>p-2$,    $\|Dv\|_{L^{2^{\star}(\gz+1)}(\Omega)}$ is
bounded by the $L^2$-norm of $(|Dv|^2+\ez)^{\gz/2}\bdz^N_{p,\ez}v$
and also the additional term $L^2$-norm of $v$. To see this,
we
rewrite the non-divergence operator $(|Dv|^2+\ez)^{\gz/2}\bdz^N_{p,\ez}v$ as
  $(\gz+2)$-Laplace operator when $\gz>p-2$ and
  $p$-Laplace operator when $\gz<p-2$.
Combining  this with
Lemma \ref{sd-es} we  conclude
  the Lemma \ref{Lq-es}  as desired.

Next we aim to deduce Theorems \ref{th-re1} and \ref{th-re1a} from Theorem \ref{thm-v}.
 To proceed, we required the following results.

\begin{thm}\label{th-re0}

Suppose that the triplet $ (n,p,\gamma)$ is admissible and $\gz>-\frac{4}{n+2}$.
Let
$\Omega$ be any bounded Lipschitz domain in $\rr^n$, and $u\in C^0(\overline \Omega)$ be  any viscosity solution
  to the problem \eqref{plapgz} with $f\in C^0(\Omega)\cap L^\fz(\Omega)$.
  The following holds.
\begin{enumerate}
  \item [(i)] If $\Omega$ is convex, then \eqref{thxx}  and \eqref{th1-1} hold.
.
  \item [(ii)] Assume that $\Omega$ is not convex and that
$\lim_{r\to 0^+}\Psi_{\Omega}(r)<\dz$ for some sufficiently small positive constant $\dz=\dz(p,n,\gz,
L_{\Omega}, d_{\Omega})$. Then \eqref{thxx} holds. Moreover,

\begin{itemize}
\item[$\bullet$]
If $\gz\le p-2$, then \eqref{th1-1} holds with the constant  $C=C(p,n,\gz,
L_{\Omega}, d_{\Omega})$.
\item[$\bullet$]
If $\gz>p-2$, then
\begin{align}\label{th1xx-2}
\|Du\|^{\gz+1}_{L^{(1+\gz)2^\star}(\Omega)}+
\|D[|Du|^{\gz}Du]\|_{L^2(\Omega)}\le C\|f \|_{L^2(\Omega)}+ C  \|u\|^{\gz+1}_{L^2(\Omega)}
\end{align}
with $C=C(p,n,\gz,
L_{\Omega}, d_{\Omega})$.
\end{itemize}
\end{enumerate}
\end{thm}

 It is standard to get  Theorems \ref{th-re0} by applying Theorem \ref{thm-v} to suitable approximation equations.
 To  guarantee that  an approximation argument works, we need an a priori global H\"older estimates when  $(1+\gamma)2^\star>n$,
 that is,  $ \|v\|_{C^\alpha(\overline \Omega)}$ is bounded via  $ \|(\bdz^N_{p,\ez}v)(|Dv|^2+\ez)^{\frac{\gz+\beta} 2}Dv\|_{L^{2}(\Omega)}$
 and also a local  H\"older estimates when  $(1+\gamma)2^\star\le n$, $ \|v\|_{C^\alpha(B)}$ is bounded
via $ \|(\bdz^N_{p,\ez}v)(|Dv|^2+\ez)^{\frac{\gz+\beta} 2}Dv\|_{L^{\infty}(2B)}$
and   $ \|(\bdz^N_{p,\ez}v)(|Dv|^2+\ez)^{\frac{\gz+\beta} 2}Dv\|_{L^{2}(\Omega)}$.
 See Section 5 for more details.
 Applying this to a standard regularized equation in smooth approximating sub-domains $\Omega_m$,
 under  $f \in C^0(  \Omega)\cap L^\fz(\Omega)$ we obtain    $ C^{0,\az}(\overline \Omega)$-regularity of
 viscosity solutions  to the Dirichlet problem \eqref{plapgz}; see Section 6 for details.

To get Theorem \ref{th-re0}, under  $f \in C^0(  \Omega)\cap L^\fz(\Omega)$, set $\bar f=f$ in $\Omega$ and $\bar f=0$ on
$\rr^n\backslash \Omega$. Also, set $\bar u=u$ in $\Omega$ and $\bar u=0$ on
$\rr^n\backslash \Omega$. Now we let $f^{\ez}=\bar f\ast \eta^{\ez}$ and $g^{\ez}=\bar u \ast \eta^{\ez}$ where
$\eta^{\ez}$ is standard mollifier.
Given any viscosity solution $u$ to \eqref{plapgz},
we work with  approximation equation used in Attouchi and Ruosteenoja \cite{ar18}:
\begin{align*}
\left\{
\begin{aligned}
-(|Du^{\ez,m}|^2+\ez)^{\frac \gz 2}
\bdz^N_{p,\ez}u^{\ez,m} +\lz u^{\ez,m}&=\lz g^\ez+ f^{\ez}
&\quad{\rm in}\ &\Omega_m\\
u^{\ez,m}&=0&\quad {\rm on}\ &\partial\Omega_m,
\end{aligned}
\right.
\end{align*}
where $\lz\in (0,1)$ and $\Omega_m$ is a smooth domain such that
$\lim_{m\to \fz}\Omega_m=\Omega$.
By applying Theorem \ref{thm-v}  to $ u^{\ez,m}$,
 thanks to the  global H\"older regularity in Section 5 and Section 6,
 we   conclude the convergence of $u^{\ez,m}$ to $u$ in several necessary function spaces,
 and then get the desired regularity \eqref{thxx} and  upper bound \eqref{th1-1}.
Note that,  since the uniqueness to \eqref{plapgz} is unavailable,
 $ \lambda>0$ is  necessary  to guarantee the convergence of $u^{\ez,m}$ to the given solution $u$.
 See section 7 for more details.

In Section 8, under  $f \in L^2(\Omega)$, we apply Theorems \ref{th-re0} to $ |Dv^{\ez}|^\gamma\Delta^N_pv^{\ez}=f\ast \eta_\epsilon$.
By verifying the desired convergence properties,
we obtain Theorem \ref{th-re1} and Theorem \ref{th-re1a}.
A key observation is that  all estimates in the approximation argument only depend on
  $\|f\ast \eta_\epsilon\|_{L^2(\Omega)}$, provided that $\Omega$ is convex or
  when $\Omega$ is not convex if $\gz<p-2$.
For a detailed discussion, see Section 8.

\section{Viscosity solutions}
Given a domain $ \Omega\subset\rn$, for any function $v\in C^2(\Omega)$,   denote by $\bdz v$ the Laplacian and
$\bdz_\fz v$ the $\fz$-Laplacian, that is,
$$\bdz v=\sum_{i=1}^nv_{x_ix_i},\quad
\bdz_\fz v=\sum_{1\le i,j\le n}v_{x_ix_j}v_{x_i}v_{x_j}.$$
The normalized $\fz$-Laplacian is given by
$\bdz^N_{\fz} v:=|Dv|^{-2}\bdz_\fz v$ if  $  Dv\ne0$.
For  $1<p<\fz$, the  normalized $p$-Laplacian  is defined as
$$\bdz^N_p v=(p-2)\bdz^N_{\fz}v+\bdz v.$$
 Their $ \ez$-regularization for $ \ez\in(0,1]$   are given by
$$\bdz^N_{\fz,\ez}v:=\frac{\bdz_\fz v}{|Dv|^2+\ez},\quad
\bdz^N_{p,\ez}v=(p-2)\bdz^N_{\fz,\ez}v+\bdz v $$

Recall the following definition of viscosity solutions introduced by Crandall-Ishii-Lions \cite{cil92}.

\begin{defn}\label{d1}

Let $ p>1$ and $\gz>-1$.  An upper
semicontinuous function $u$ is said to be a viscosity
subsolution   to
$-|Du|^\gz\Delta_p^Nu=f$ in $ \Omega$
provided that  at  any $ x_0\in\Omega$,   for any   $\phi\in C^2(\Omega)$ with  $u-\phi$ reaching its local maximum at point $x_0$ such that the following holds:

\begin{enumerate}
\item[(i)] if $ \gz>0$,  one has
\begin{align*}
&-|D\psi(x_0)|^{\gz}\bdz^N_p\psi(x_0)\le f(x_0).
\end{align*}
\item[(ii)]if $ \gz=0$,  one has
\begin{equation*}\lf\{\begin{array}{rcl}
&\displaystyle-\bdz^N_p \phi(x_0)\le f(x_0),&\quad {\rm if}\  D\phi(x_0)\neq 0,\\
&\displaystyle -(p-2)\max_{|\eta|=1}\langle D^2\phi(x_0)
\eta,\eta\rangle-\bdz \phi(x_0)\le f(x_0),&\quad {\rm if}\ D\phi(x_0)=0\
{\rm and}\ p\ge 2,\\
&\displaystyle-(p-2)\min_{|\eta|=1}\langle D^2\phi(x_0)
\eta,\eta\rangle-\bdz \phi(x_0)\le f(x_0),&\quad {\rm if}\ D\phi(x_0)=0\
{\rm and}\  1<p<2,
\end{array}\r.
\end{equation*}
\item[(iii)] if $ -1<\gz<0$ and $D\psi(x)\neq 0$ for
$x\neq x_0$, then
$$\lim_{r\to 0}\inf_{x\in B_r(x_0)\backslash\{x_0\}}[-|D\psi(x)|^{\gz}\bdz^N_p\psi(x)]\le f(x_0).$$
\end{enumerate}

  An lower
semicontinuous function $u$ is said to be a viscosity
supsolution   to the equation
$-|Du|^\gz\Delta_p^Nu=f$ in $ \Omega$
if $-u$ is its viscosity
subsolution.

A function $u\in C^0(\Omega)$ is a viscosity solution  to
$-|Du|^\gz\Delta_p^Nu=f$ in $ \Omega$  if it is both viscosity subsolution
and supersolution.
\end{defn}


We have the following remark in the case $-1<\gz<0$.
\begin{rem}\rm
  In the case $-1<\gz<0$,  the above  of viscosity subsolution to $-|Du|^\gz\Delta_p^Nu=f$ in $ \Omega$   is equivalent to the one proposed by Birindelli and Demengel
 \cite{bd04}, that is,
  at  any $ x_0\in\Omega$,   either $$-|D\psi(x_0)|^{\gz}\bdz^N_p\psi(x_0)\le f(x_0) $$
whenever  $\psi\in C^2(\Omega)$ such that $u-\psi$ attains its maximum at $x_0$ and $D\psi(x_0)\neq 0$,  or  $u=constant$   and $f\ge 0 $ in
  $B_{\dz}(x_0)$ for some $ \dz>0$.
  We refer
Attouchi-Ruosteenoja \cite{ar18} for more details.
\end{rem}

Finally, we   recall that the following inequalities for vectors
 (see for example Lindqvist \cite{lin19}), which will be used in Section 7 and Section 8.

\begin{lem}\label{in-vec}
Let $\gz>-1$. For any $a,b\in \rr^n$, then:
\begin{itemize}
\item if $-1<\gz<0$ and $\ez\in (0,1]$,
\begin{align}\label{in-vec1}
&\left((|a|^2+\ez)^{\frac{\gz}{2}}a-(|b|^2+\ez)^{\frac{\gz}{2}}b\right)\cdot(a-b)\ge
 (\gz+1)6^{\frac{\gz}{2}}(|a|^2+|b|^2+\ez)^{\frac{\gz}{2}}|a-b|^2.
\end{align}
In addiction, if $\ez=0$,
\begin{align}\label{in-vec2}
&\left(|a|^{\gz}a-|b|^{\gz}b\right)\cdot(a-b)\ge
 (\gz+1)6^{\frac{\gz}{2}}(|a|^2+|b|^2+1)^{\frac{\gz}{2}}|a-b|^2.
\end{align}
\item if $\gz\ge 0$ and $\ez\in (0,1]$,
\begin{align}\label{in-vec3}
&\left((|a|^2+\ez)^{\frac{\gz}{2}}a-(|b|^2+\ez)^{\frac{\gz}{2}}b
\right)\cdot(a-b)
\ge 2^{-2-\gz}(|a|^2+|b|^2+\ez)^{\frac{\gz}{2}}|a-b|^2,
\end{align}
In addition, if $\ez=0$,
\begin{align}\label{in-vec4}
&\left(|a|^{\gz}a-|b|^{\gz}b
\right)\cdot(a-b)
\ge 2^{-2-\gz}(|a|^2+|b|^2)^{\frac{\gz}{2}}|a-b|^2.
\end{align}
\end{itemize}
\end{lem}

\section{A key inequality  involving  the  operator $\bdz^N_{p,\ez}$}

Assume for this section that $\Omega\subset \rr^n$ is smooth bounded domain and that
$v\in C^\fz(\overline \Omega)$. We  establish the following algebraic inequality,  which is a generalization of
a differential inequality in \cite[Lemma 3.1]{cm18}.
\begin{lem}\label{key-in}
Let  $ p>1$ and let $\gz>-1$ satisfy \eqref{pgz}.
For all $\ez\in (0,1)$, it holds that
\begin{align}\label{key-ineq}
 C_1 |D^2 v |^2(|D v |^2+\ez)^{\gz}
& \le   C_2(\bdz^N_{p,\ez} v )^2(|D v |^2+\ez)^{  \gz }\nonumber \\
&\quad\quad
+
{\rm div}\left\{(|D v |^2+\ez)^{\gz}(D^2 v D v -\bdz  v D v )\right\}\ \mbox{
a.e. in $\Omega$,}
\end{align}
where $C_1=C_1(p,n,\gz) $ and $C_2=C_2(p,n,\gz) $ are positive  constants.
\end{lem}

To prove  Lemma \ref{key-in}, we need  the following lemma; see \cite{ss22} by Sarsa.
For the sake of completeness, we give its proof here.
\begin{lem}\label{fun-in}
 It holds that
\begin{align}\label{fin}
|D^2v|^2-2|D|Dv||^2+(\bdz^N_{\fz}v)^2\ge \frac 1{n-1}(\bdz v-\bdz^N_{\fz}v)^2
\quad {\rm in}\quad \Omega\backslash\{Dv\neq 0\}.
\end{align}
\end{lem}
\begin{proof}
Assume that $Dv(x_0)\neq 0$ for $x_0\in \Omega$.  Up to considering
$\frac{Dv(x_0)}{|Dv(x_0)|}$, we may assume   $|Dv(x_0)|=1$.
Up to some rotation,
we further assume that $Dv(x_0)=(0,...,0,1)$.
 At point
$x_0$, one then has
\begin{align}\label{fun-eq1}
|D|Dv||^2=\sum_{i=1}^nv_{x_ix_n}^2,\quad \bdz^N_\fz v=v_{x_nx_n}.
\end{align}

Observe that
\begin{align*}
|D^2v|^2&=\sum_{i=1}^nv_{x_ix_i}^2+2\sum_{1\le i<j\le n}v_{x_ix_j}^2\ge -v_{x_nx_n}^2+2\sum_{i=1}^{n}v_{x_ix_n}^2+\sum_{i=1}^{n-1}v_{x_ix_i}^2.
\end{align*}
Via the Cauchy-Schwartz inequality  we obtain
\begin{align*}
|D^2v|^2&\ge -v_{x_nx_n}^2+2\sum_{i=1}^nv_{x_ix_n}^2+\frac 1{n-1}
\left(\sum_{i=1}^{n-1}v_{x_ix_i}\right)^2\\
&=-v_{x_nx_n}^2+2\sum_{i=1}^nv_{x_ix_n}^2+\frac 1{n-1}
\left(\bdz v-v_{x_nx_n}\right)^2.
\end{align*}
From this and  \eqref{fun-eq1} we conclude \eqref{fin}  as desired.
\end{proof}

 Using the above lemma, we have the following structural inequality.
\begin{lem}\label{key-in2}
Let $\ez\in (0,1)$ and $\eta>0$.
Then, it holds that
\begin{align}\label{in2-eq}
&|D^2 v |^2+2\gz\frac{|D^2 v D v |^2}{|D v |^2
+\ez}+[\gz^2-(\gz-p+2)^2](\bdz^N_{\fz,\ez} v )^2
\nonumber \\
&\ge2(p-1)\left[\gz+1-(p-1)\frac{n-2}{2(n-1)}\right]
(\bdz^N_{\fz,\ez}  v )^2-\frac{C(p)}{\eta}(\bdz^N_{p,\ez} v )^2
-\eta|D^2 v |^ 2
\end{align}
almost everywhere in $\Omega$. 
\end{lem}

\begin{proof}[Proof of Lemma \ref{key-in2}]

Note that, $D^2 v =0$ a.e in $\{D v =0\}$ and hence
$ \Delta^N_{p,\ez}v$ a.e on $\{D v =0\}$.
Then \eqref{in2-eq} holds a.e on $\{D v =0\}$.
 Suppose now that $D v \neq 0$.
Set
$${\bf L}:=|D^2 v |^2+2\gz\frac{|D^2 v  D v |^2}{
|D v |^2+\ez}+\gz^2(\bdz^N_{\fz,\ez} v )^2.$$
For all $\eta>0$ we claim that
\begin{align}\label{in2-1}
{\bf L}\ge [\frac{(p-1)^2}{n-1}+(\gz+1)^2]\left(
\bdz^N_{\fz,\ez} v \right)^2 -\frac{C(p)}{\eta} (\bdz^N_{p,\ez} v)^2
-\eta|D^2 v |^2.
\end{align}
If this claim holds for the moment, a direct calculation shows that
\begin{align*}
&{\bf L}-(\gz-p+2)^2\left(
\bdz^N_{\fz,\ez} v \right)^2\\
&\ge[\frac{(p-1)^2}{n-1}+(\gz+1)^2-(\gz+p-2)^2]\left(
\bdz^N_{\fz,\ez} v \right)^2 -\frac{C(p)}{\eta} (\bdz^N_{p,\ez} v)^2
-\eta|D^2 v |^2.
\end{align*}
Observe that
\begin{align*}
\frac{(p-1)^2}{n-1}+(\gz+1)^2-(\gz-p+2)^2
&=\frac{(p-1)^2}{n-1}+(p-1)(2\gz-p+3)\\
&=(p-1)\left[\frac{p-1}{n-1}+2\gz-p+3\right]\\
&=(p-1)\left[\frac{p-1}{n-1}-(p-1)+2(\gz+1)\right]\\
&=2(p-1)\left[\gz+1-(p-1)\frac{n-2}{2(n-1)}\right].
\end{align*}
From this we conclude this result.

We now prove the claim \eqref{in2-1}, which is divided into the following two steps.

{\it Step 1}. For all $\eta\in (0,1)$,  we show that
\begin{align}\label{cla-st1}
{\bf L}&\ge \left[\frac{(p-1)^2}{n-1}+(\gz+1)^2\right](\bdz^N_{\fz,\ez} v )^2\nonumber\\
&\quad\quad
+\Big\{\frac n{n-1}[(\bdz_\fz^N  v )^2-(\bdz^N_{\fz,\ez} v )^2]\nonumber\\
&\quad\quad+2\gz[\frac{|D^2 v  D v |^2}{
|D v |^2+\ez}-(\bdz^N_{\fz,\ez} v )^2]
+2[\frac{|D^2v Dv |^2}{|Dv |^2}-(\bdz^N_\fz  v )^2]\nonumber\\
&\quad\quad+ \frac{2(p-2)}{n-1}  ( \bdz_{\fz,\ez}^N v)^2 \frac \ez{|D v |^2}\Big\}
-\frac {C(p,n)}{\eta}( \bdz^N_{p,\ez} v)^2-\eta |D^2 v |^2.
\end{align}

By $D v \neq 0$,  Lemma \ref{fun-in} gives
\begin{align*}
|D^2 v |^2\ge \frac{(\bdz  v -\bdz^N_\fz  v )^2}{n-1}+2
[|D|D v ||^2-(\bdz^N_{\fz} v )^2]+(\bdz_\fz^N  v )^2.
\end{align*}
 Noting
$$\gz^2(\bdz^N_{\fz,\ez} v )^2=(\gz+1)^2(\bdz^N_{\fz,\ez} v )^2-2\gz(\bdz^N_{\fz,\ez} v )^2
-(\bdz^N_{\fz,\ez} v )^2,$$
we obtain
\begin{align}\label{in2-2}
{\bf L}&\ge\frac 1{n-1}(\bdz  v -\bdz^N_\fz  v )^2
+(\gz+1)^2(\bdz^N_{\fz,\ez} v )^2
+[(\bdz_\fz^N  v )^2-(\bdz^N_{\fz,\ez} v )^2]
\nonumber\\
&\quad+2\gz[\frac{|D^2 v  D v |^2}{
|D v |^2+\ez}-(\bdz^N_{\fz,\ez} v )^2] +2[\frac{|D^2 v  D v |^2}{|D v |^2}-(\bdz^N_\fz  v )^2]
\end{align}
Now we will  bound the first term of the right hand side in \eqref{in2-2}.
Recalling
$$\bdz  v =-(p-2)\bdz^N_{\fz,\ez} v
+ \bdz^N_{p,\ez} v,$$
one has
\begin{align}\label{in2-3}
(\bdz  v -\bdz^N_\fz  v )^2
&=\left(-(p-2)\bdz^N_{\fz,\ez} v + \bdz^N_{p,\ez} v
-\bdz^N_\fz  v \right)^2\nonumber\\
&=\big[(p-2)^2\bdz^N_{\fz,\ez} v +(\bdz^N_\fz  v )^2
+2(p-2)\bdz_\fz^N v\bdz_{\fz,\ez}^N v\big]\nonumber\\
&\quad+\big[ -2  \bdz^N_{p,\ez} v[(p-2)\bdz^N_{\fz,\ez} v
+\bdz^N_{\fz} v ]+( \bdz^N_{p,\ez} v)^2\big].
\end{align}
Since
$(p-2)^2=(p-1)^2-2(p-2)-1$,
we write the first term of right-hand side of \eqref{in2-3} as
\begin{align}\label{in2-4}
&(p-2)^2(\bdz^N_{\fz,\ez} v )^2  +(\bdz^N_\fz  v )^2+2(p-2)\bdz_\fz^N v\bdz_{\fz,\ez}^N v\nonumber\\
&
=(p-1)^2(\bdz^N_{\fz,\ez} v )^2
-2(p-2)(\bdz^N_{\fz,\ez} v )^2
 -(\bdz^N_{\fz,\ez} v )^2\nonumber \\
&\quad+ (\bdz^N_\fz  v )^2+2(p-2)\bdz_\fz^N v\bdz_{\fz,\ez}^N v\nonumber\\
&= (p-1)^2(\bdz^N_{\fz,\ez} v )^2+ [ (\bdz^N_\fz  v )^2-(\bdz^N_{\fz,\ez} v )^2]\nonumber\\
&\quad+2(p-2)\bdz^N_{\fz,\ez} v[\bdz_\fz^N v-\bdz^N_{\fz,\ez} v ]
\end{align}

  Since
 $$ \frac1{|Dv|^2}=  \frac1{|Dv|^2+\ez}\frac\ez{ |Dv|^2}+ \frac1{|Dv|^2+\ez},$$
 we have
 $$ \bdz_\fz^N v-\bdz^N_{\fz,\ez} v
  = \bdz_{\fz,\ez}^N v  \frac \ez{|D v |^2}  .$$
Thus
 $$
 2(p-2)\bdz^N_{\fz,\ez} v[\bdz_\fz^N v-\bdz^N_{\fz,\ez} v ]
 =  2(p-2) ( \bdz_{\fz,\ez}^N v)^2 \frac \ez{|D v |^2} .$$
  For the second term of right-hand side in \eqref{in2-3}, via Young's inequality we have
 \begin{align*}
-2  \bdz^N_{p,\ez} v[(p-2)\bdz^N_{\fz,\ez} v
+\bdz^N_{\fz} v ]+( \bdz^N_{p,\ez} v)^2\ge
-\frac {C(p)}{\eta(n-1)}( \bdz^N_{p,\ez} v)^2-\eta(n-1) |D^2 v |^2.
\end{align*}

 We then have

\begin{align*}
\frac1{n-1}(\bdz  v -\bdz^N_\fz  v )^2
&\ge \frac{ (p-1)^2}{n-1}(\bdz^N_{\fz,\ez} v )^2+ \frac1{n-1}[(\bdz_\fz^N  v )^2-(\bdz^N_{\fz,\ez} v )^2]\\
&\quad\quad+ \frac{2(p-2)}{n-1}  ( \bdz_{\fz,\ez}^N v)^2 \frac \ez{|D v |^2} -\frac {C(p,n)}{\eta}( \bdz^N_{p,\ez} v)^2-\eta |D^2 v |^2.
\end{align*}
Plugging this into    \eqref{in2-2},  and by
 applying Young's inequality we conclude \eqref{cla-st1}.

{\it Step 2}. We prove that ${\bf T}\ge 0$. Here ${\bf T}$ is defined by
\begin{align*}
{\bf T}:&=\frac n{n-1}[(\bdz_\fz^N  v )^2-(\bdz^N_{\fz,\ez} v )^2]\\
&\quad+2\gz[\frac{|D^2 v  D v |^2}{
|D v |^2+\ez}-(\bdz^N_{\fz,\ez} v )^2]
+2[\frac{|D^2u Dv |^2}{|Dv |^2}-(\bdz^N_\fz  v )^2]\\
&\quad\quad+ \frac{2(p-2)}{n-1}  ( \bdz_{\fz,\ez}^N v)^2 \frac \ez{|D v |^2}
\end{align*}

In order to proving ${\bf T}\ge 0$, recalling that
$$\gz>-1+(p-1)\frac{n-2}{2(n-1)}\ge -1$$
and using the fact that
$$\frac{|D^2 v  D v |^2}{
|D v |^2+\ez}-(\bdz^N_{\fz,\ez} v )^2
=\frac{|D^2 v  D v |^2}{
|D v |^2+\ez}-\frac{(\bdz_\fz  v )^2}{(|D v |^2+\ez)^2}\ge 0,$$
one gets
$$
2\gz[\frac{|D^2 v  D v |^2}{
|D v |^2+\ez}-(\bdz^N_{\fz,\ez} v )^2]\ge-2[\frac{|D^2 v  D v |^2}{
|D v |^2+\ez}-(\bdz^N_{\fz,\ez} v )^2].$$
 Thus
 \begin{align*}
&  2\gz[\frac{|D^2 v  D v |^2}{
|D v |^2+\ez}-(\bdz^N_{\fz,\ez} v )^2]
+2[\frac{|D^2u Dv |^2}{|Dv |^2}-(\bdz^N_\fz  v )^2]\\
&\ge  -2[\frac{|D^2 v  D v |^2}{
|D v |^2+\ez}-(\bdz^N_{\fz,\ez} v )^2]+2[\frac{|D^2u Dv |^2}{|Dv |^2}-(\bdz^N_\fz  v )^2]\\
&= -2[(\bdz^N_\fz  v )^2-(\bdz^N_{\fz,\ez} v )^2]+  2[\frac{|D^2u Dv |^2}{|Dv |^2}- \frac{|D^2 v  D v |^2}{
|D v |^2+\ez}].
\end{align*}
We therefore obtain
\begin{align*}
{\bf T}&\ge -\frac {n-2}{n-1}[(\bdz_\fz^N  v )^2-(\bdz^N_{\fz,\ez} v )^2]   + 2[\frac{|D^2u Dv |^2}{|Dv |^2}- \frac{|D^2 v  D v |^2}{
|D v |^2+\ez}] + \frac{2(p-2)}{n-1}  ( \bdz_{\fz,\ez}^N v)^2 \frac \ez{|D v |^2}.
\end{align*}
Observe  that
 \begin{align*}-\frac {n-2}{n-1}[(\bdz_\fz^N  v )^2-(\bdz^N_{\fz,\ez} v )^2]
& =-\frac{n-2}{n-1}(\bdz^N_{\fz } v )^2[1-  \frac{|Dv|^4}{(|Dv|^2+\ez)^2}]  \\
&=-\frac{n-2}{n-1}  (\bdz^N_{\fz } v )^2 \frac{\ez^2+2\ez|Dv|^2}{(|Dv|^2+\ez)^2 }.
\end{align*}
By Cauchy-Schwarz inequality, one has
 \begin{align*}
 2[\frac{|D^2u Dv |^2}{|Dv |^2}- \frac{|D^2 v  D v |^2}{
|D v |^2+\ez}] &= 2 \frac{|D^2u Dv |^2}{|Dv |^2}[1-  \frac{| D v |^2}{
|D v |^2+\ez}]\\
&= 2 \frac{|D^2u Dv |^2}{|Dv |^2}\frac\ez{|Dv |^2+\ez}\\
&\ge 2 ( \Delta^N_\fz v)^2\frac\ez{|Dv |^2+\ez}.
\end{align*}
Noting
 \begin{align*}
 \frac{2(p-2)}{n-1}  ( \bdz_{\fz,\ez}^N v)^2 \frac \ez{|D v |^2}=   \frac{2(p-2)}{n-1} ( \bdz_{\fz}^N v)^2 \frac{ \ez |Dv|^2}{(|D v |^2+\ez)^2},
\end{align*}
we further have
\begin{align*}
{\bf T}&\ge  \frac{ (\bdz^N_{\fz } v )^2}{(|Dv|^2+\ez)^2} \frac{1}{n-1} \left\{
 -(n-2)  [\ez^2+2\ez|Dv|^2]+ 2(n-1)\ez[|Dv |^2+\ez]+  2(p-2)  \ez |Dv|^2\right\}\\
 &= \frac{ (\bdz^N_{\fz } v )^2}{(|Dv|^2+\ez)^2} \frac{1}{n-1} \{[2(n-1)-(n-2)]\ez^2+[2(n-1)-2(n-2)+2(p-2)]\ez |Dv |^2\}\\
 &=  \frac{ (\bdz^N_{\fz } v )^2}{(|Dv|^2+\ez)^2} \frac{1}{n-1}\left\{ n\ez^2+2(p-1)\ez|Dv |^2\right\}\\
 &\ge0.
 \end{align*}

Finally, thanks to the step 1 and step 2, we conclude claim as desired and hence
we finish this proof.
\end{proof}

We now begin to
show Lemma \ref{key-in}.
\begin{proof}[Proof of Lemma \ref{key-in}]
Applying a differential inequality in \cite[Lemma 3.1]{cm18} we deduce
\begin{align}\label{pr2-1}
&(|D v |^2+\ez)^{  \gz  }\left\{|D^2 v |^2+2\gz\frac{|D^2 v D v |^2}{|D v |^2
+\ez}+(\gz^2-(\gz-p+2)^2)(\bdz^N_{\fz,\ez} v )^2\right\}
\nonumber \\
&=(|D v |^2+\ez)^{  \gz  }\left[(\bdz^N_{p,\ez} v)^2-2(p-2-\gz)
\bdz^N_{p,\ez} v \bdz^N_{\fz,\ez} v\right]\nonumber\\
&\quad
+{\rm div}\left((|D v |^2+\ez)^{  \gz  }(D^2 v D v -\bdz  v D v )\right)\nonumber\\
&\le [1+\frac1\dz ](|D v |^2+\ez)^{  \gz  }(\bdz^N_{p,\ez} v)^2+(\gz-p+2)^2\dz(|D v |^2+\ez)^{  \gz  } (\bdz^N_{\fz,\ez} v )^2\nonumber\\
&\quad+{\rm div}\left((|D v |^2+\ez)^{  \gz  }(D^2 v D v -\bdz  v D v )\right),
\end{align}
where we apply Young's inequality for the term $-2(p-2-\gz)
\bdz^N_{p,\ez} v \bdz^N_{\fz,\ez} v$  so to get in the last inequality.
Observe that the left-hand side of \eqref{pr2-1} enjoys a lower bound by
using Lemma \ref{key-in2}.
 From this and  \eqref{pr2-1} again it follows that
\begin{align*}
&\left\{ 2(p-1)\left[\gz+1-(p-1)\frac{n-2}{2(n-1)}\right]  -(\gz-p+2)^2\dz\right\}
(|D v |^2+\ez)^{  \gz  } (\bdz^N_{\fz,\ez}  v )^2
\nonumber \\
&\quad\le (1+ \frac1\dz+C(p)\frac1\eta  )(|D v |^2+\ez)^{  \gz  }(\bdz^N_{p,\ez} v)^2\\
&\quad\quad+{\rm div}\left([|D v |^2+\ez]^{  \gz  }(D^2 v D v -\bdz  v D v )\right)+\eta|D^2 v |^2
\end{align*}
for all $\eta>0$.

Since $\gz>-1+(p-1)\frac{n-2}{2(n-1)}$, we can choose $\dz>0$ such that
$$ (p-1)[ \gz+1-(p-1)\frac{n-2}{2(n-1)}]-(\gz-p+2)^2\dz:=c=c(p,n,\gz)>0,$$
which leads to
\begin{align}\label{pr2-2}
 c(|D v |^2+\ez)^{  \gz  } (\bdz^N_{\fz,\ez}  v )^2
&\le C(1+\frac1\eta)(|D v |^2+\ez)^{  \gz  }(\bdz^N_{p,\ez} v)^2
+ \eta(|D v |^2+\ez)^{  \gz  }|D^2 v |^2\nonumber\\
&\quad\quad\quad\quad\quad\quad+{\rm div}\left((|D v |^2+\ez)^{  \gz  }(D^2 v D v -\bdz  v D v )\right).
\end{align}
for some $C=C(p,n,\gz)>0$.

On the other hand, we take $\dz=1$ in the inequality \eqref{pr2-1} to obtain
\begin{align*}
& [|D v |^2+\ez]^{  \gz  }\left\{|D^2 v |^2+2\gz\frac{|D^2 v D v |^2}{|D v |^2
+\ez}+\gz^2(\bdz^N_{\fz,\ez} v )^2\right\}
\nonumber \\
&\le 2(|D v |^2+\ez)^{  \gz  }(\bdz^N_{p,\ez} v)^2
+5(r-p+2)^2(\bdz^N_{\fz,\ez} v )^2(|Dv|^2+\ez)^{\gz}\\
&\quad
+{\rm div}\left([|D v |^2+\ez]^{  \gz  }(D^2 v D v -\bdz  v D v )\right).
\end{align*}
Since the second term  of the right-hand side above can be bounded by the right-hand side of \eqref{pr2-2}, then we conclude that
\begin{align*}
& (|D v |^2+\ez)^{  \gz  }\left\{|D^2 v |^2+2\gz\frac{|D^2 v D v |^2}{|D v |^2
+\ez}+\gz^2(\bdz^N_{\fz,\ez} v )^2-C_4\eta|D^2 v |^2\right\}
\nonumber \\
&\le (C_1+C_2\eta^{-1})(|D v |^2+\ez)^{  \gz  }(\bdz^N_{p,\ez} v)^2+C_3{\rm div}\left((|D v |^2+\ez)^{  \gz  }(D^2 v D v -\bdz  v D v )\right),
\end{align*}
where $C_i=C_i(p,n,\gz)$ with $i=1,2,3,4$. Recalling that (see \cite[Theorem 2.1]{bcdm22})
$$|D^2 v |^2+2\gz\frac{|D^2 v D v |^2}{|D v |^2
+\ez}+\gz^2(\bdz^N_{\fz,\ez} v )^2\ge \min\{1,(\gz+1)^2\}|D^2v|^2.$$
Thus
\begin{align*}
&[ \min\{1,(\gz+1)^2\}-C_4\eta](|D v |^2+\ez)^{  \gz  }|D^2v|^2
\nonumber \\
&\le [C_1+C_2\eta^{-1}](|D v |^2+\ez)^{  \gz  }(\bdz^N_{p,\ez} v)^2
+C_3{\rm div}\left((|D v |^2+\ez)^{  \gz  }(D^2 v D v -\bdz  v D v )\right).
\end{align*}
Therefore, we choose a suitable $\eta>0$ so that
$C_4\eta=\frac 12 \min\{1,(\gz+1)^2\}$ as desired.

\end{proof}

\section{Global Sobolev estimates and proof of Theorem \ref{thm-v}}
In this section, we assume that the following:
\begin{itemize}
\item[$\bullet$] $\Omega$ is smooth bounded domain
in $\rr^n$ with $n\ge 2$.
\item[$\bullet$] $v\in C^\fz(\overline \Omega)$ and
$v=0$ on $\partial \Omega$.
\item[$\bullet$] Let $\ez\in (0,1]$. Let $p>1$ and let $\gz>-1$ satisfy
\eqref{pgz}.
\end{itemize}
We shall establish the global second order estimate for the gradient term
$(|Dv|^2+\ez)^{\gz}Dv$. As a consequence, some higher regularity of
gradient and H\"older regularity for $v$ are also proved in this section.
Moreover, Theorem \ref{thm-v} follows from these key estimate.

The following global second order estimate is key point in this section.
\begin{lem}\label{sd-es}

Let $\bz\ge0$. Then the following holds.

\begin{enumerate}
  \item [(i)] If $\Omega$ is convex,  then
\begin{align}\label{sd-es1}
\int_{\Omega}|D^2 v |^2(|D v |^2+\ez)^{\gz+\bz}\,dx
\le  C (1+\bz^2)\int_{\Omega}
(\bdz^N_{p,\ez}v)^2(|D v |^2+\ez)^{\gz+ \bz}\,dx,
\end{align}
where $C=C(p,n,\gz)>0$.
  \item [(ii)] Assume that $\Omega$ is not convex. There exists a positive constant $
  \dz=\dz(p,n,\gz, d_{\Omega},L_{\Omega})$ such that if $\lim_{
  r\to 0^{+}}\Psi_{\Omega}(r)\le \delta $,  then for some constat
$C =C (p,n,\gz, L_{\Omega},d_{\Omega})$
 \begin{align}\label{sd-es2}
&\int_{\Omega}|D^2v|^2(|Dv|^2+\ez)^{\gz+\beta }\,dx\nonumber\\
&
\le C(1+\bz^2)\left[\int_{\Omega}
(\bdz^N_{p,\ez}v)^2 (|Dv|^2+\ez)^{\gz+  \beta  }\,dx
+ \left(\int_{\Omega}(|Dv|^2+\ez)^{\frac{\gz + \beta+1 }{2}}  \,dx\right)^2\right].
\end{align}
\end{enumerate}

\end{lem}

We denote
\begin{align}\label{q-star}
2^{\star}=q>2\quad{\rm for\ all }\quad q>2 \quad {\rm if}\quad  n=2\quad {\rm and}\quad
2^{\star}=\frac{2n}{n-2}\quad{\rm if}\quad n\ge 3.
\end{align}
From this second order estimate, we can show the following higher gradient estimate.
Observe that, if $\gz=p-2$, the following lemma follows from \cite{cm17}.
\begin{lem}\label{Lq-es}

\begin{enumerate}
  \item [(i)] If $ \Omega$ is convex, then
\begin{align}\label{lq-es1}
\||Dv|\|_{L^{2^{\star}(\gz+1)}(\Omega)}
\le C\|(|Dv|^2+\ez)^{ \gz/2 }\bdz^N_{p,\ez}v\|^{\frac 1{\gz+1}}_{L^2(\Omega)}
+C\ez^{\frac 12}
\end{align}
for some constant $C=C(p,n,\gz)$.

  \item [(ii)]
Assume that $\Omega$ is not convex and that $\lim_{
  r\to 0^{+}}\Psi_{\Omega}(r)\le \delta $ for a sufficiently small positive constant $
  \dz=\dz(p,n,\gz, d_{\Omega},L_{\Omega})$.

  If $\gz\le p-2$, then
\begin{align}\label{lq-es3}
\||Dv|\|_{L^{2^{\star}(\gz+1)}(\Omega)}
\le C\|(|Dv|^2+\ez)^{ \gz/2 }\bdz^N_{p,\ez}v\|^{\frac 1{\gz+1}}_{L^2(\Omega)}
+C\ez^{\frac 12},
\end{align}
where $C=C(p,n,\gz, L_{\Omega}, d_{\Omega})$.

If $ \gamma>p-2$, then
\begin{align}\label{lq-es2}
\||Dv|\|_{L^{2^{\star}(\gz+1)}(\Omega)}
\le C\left[\|(|Dv|^2+\ez)^{  \gz/2 }\bdz^N_{p,\ez}v\|^{\frac 1{\gz+1}}_{L^2(\Omega)}
+\|v\|_{L^2(\Omega)}\right]
+C\ez^{\frac 12},
\end{align}
where $C=C(p,n,\gz, L_{\Omega}, d_{\Omega})$.


\end{enumerate}

\end{lem}
Observe that, from Lemmas \ref{sd-es} and \ref{Lq-es}, we can prove
Theorem \ref{thm-v}.
\begin{proof}
[Proof of Theorem \ref{thm-v}]
This is a direct consequence of Lemmas \ref{sd-es} and \ref{Lq-es}.

\end{proof}

\subsection{Proof of Lemma \ref{sd-es}}
We begin with recalling the following  weighted trace inequality  as proved by  Cianchi-Maz'ya \cite{cm19} and see also \cite[Proposition 6.3]{accfm}.
\begin{lem}\label{trace-sob}
 For any
$x\in \partial \Omega$, there exists $R_{\Omega}>0$ depending only on $\Omega$ such that
for all $r\in (0, R_{\Omega})$  and for all $v\in W^{1,2}_0(B_r(x))$
\begin{equation}\label{case1-1}
  \int_{\partial\Omega\cap B_r(x)}v^2|\mathcal{B}|\,d \mathcal{H}^{n-1}(y)\leq C  \Psi_{\Omega}  (r)
\int_{\Omega\cap B_r(x)}|\nabla v|^2\,dy,
\end{equation}
where constant $C$ only depends on $L_{\Omega}$.

\end{lem}

We now prove Lemma \ref{sd-es} by an argument of \cite{cm19}. We begin to
show the following lemma.
\begin{lem} For any cut-off function $\xi\in C^\fz_c(\rr^n)$, it holds that
 \begin{align}\label{sd-eq7}
&\frac 12C_1
\int_{\Omega}|D^2 v |^2(|D v |^2+\ez)^{ \gz+\bz}
\xi^2\,dx\nonumber \\
&\le C (1+\bz^2)
\int_{\Omega} (|Dv|^2+\ez)^{\gz}(\Delta_{p,\ez}^Nu)^2(|D v |^2+\ez)^{\bz}\xi^2\,dx+C \int_{\Omega}(|D v |^2+\ez)^{ \gz+\bz+1 }|D\xi|^2\,dx\nonumber\\
&\quad\quad+
\int_{\partial\Omega}(|D v |^2+\ez)^{ \gz+\bz}|D v |^2|\mathcal{B}|\xi^2\,d\mathcal{H}^{n-1}.
\end{align}
\end{lem}

\begin{proof}
Multiplying both sides
 of the inequality \eqref{key-in} by $\xi^2(|D v |^2+\ez)^{\bz}$ for  any cut-off function $\xi\in C^\fz_c(\rr^n)$,
and  integrating   over $\Omega$, we obtain
 \begin{align}\label{sd-eq2}
&C_1
\int_{\Omega}|D^2 v |^2(|D v |^2+\ez)^{\gz+\bz}
\xi^2\,dx\nonumber \\
&\le C_2
\int_{\Omega} ((|Dv|^2+\ez)^{\frac\gz2}\Delta_{p,\ez}^Nu)^2(|D v |^2+\ez)^{\bz}
\xi^2\,dx\nonumber\\
&\quad+
\int_{\Omega}{\rm div}\left((|D v |^2+\ez)^{\gz}(D^2 v D v -\bdz  v D v )\right)(|D v |^2+\ez)^{\bz}\xi^2\,dx.
\end{align}
Via integration by
 parts we get
  \begin{align}\label{sd-eq3}
&\int_{\Omega}{\rm div}\left((|D v |^2+\ez)^{\gz}(D^2 v D v -\bdz  v D v )\right) (|D v |^2+\ez)^{\bz}\xi^2\,dx\nonumber \\
&=-2\int_{\Omega}\xi(|D v |^2+\ez)^{ \gz+\bz}(D^2 v D v -\bdz  v D v )
\cdot D\xi\,dx\nonumber \\
&\quad+\int_{\partial \Omega}(|D v |^2+\ez)^{ 2\gz+\bz}(D^2 v D v -\bdz  v D v )\cdot \nu\xi^2\,d\mathcal{H}^{n-1}\nonumber \\
&\quad-\beta \int_{\Omega}[|D^2vDv|^2-\bdz v\bdz_\fz v](|D v |^2+\ez)^{ \gz+\bz-1}\xi^2\,dx
\end{align}
For the last term of right-hand side in \eqref{sd-eq3},
via identity $\bdz^N_{p,\ez} v=\bdz v+(p-2)\bdz^N_{\fz,\ez}v$, Young's inequality
 with $\eta>0$ and $\bz\ge 0$ we have
\begin{align}\label{sd-eq4}
&-\beta \int_{\Omega}[|D^2vDv|^2-\bdz v\bdz_\fz v](|D v |^2+\ez)^{ \gz+\bz-1}\xi^2\,dx\nonumber\\
&=-\beta \int_{\Omega}[|D^2vDv|^2+(p-2)\frac{(\bdz_\fz v)^2}{
|Dv|^2+\ez}](|D v |^2+\ez)^{ \gz+\bz-1}\xi^2\,dx\nonumber\\
&\quad-\beta \int_{\Omega}(|Dv|^2+\ez)^{\frac{\gz+\bz} 2}\bdz^N_{p,\ez}v\bdz_\fz v(|D v |^2+\ez)^{\frac{\gz+ \bz-2}{2}}\xi^2\,dx\nonumber\\
&\le
\frac14C_1\int_{\Omega}|
D^2v|^2(|D v |^2+\ez)^{ \gz+\beta }\xi^2\,dx\nonumber\\
&\quad+ C \beta^2\int_{\Omega}((|Dv|^2+\ez)^{\frac \gz 2}\bdz^N_{p,\ez}v)^2(|D v |^2+\ez)^{\bz}\xi^2\,dx.
\end{align}

For the second term of right-hand side in \eqref{sd-eq3},
by $ v =0$ on $\partial\Omega$ we can employ \cite[Equation (3, 1, 1, 2)]{gr85} to get
\begin{align*}
\big|(D^2 v D v -\bdz  v D v )\cdot \nu\big|
\le |D v |^2|\mathcal{B}|\quad{\rm on}\quad \partial\Omega.
\end{align*}
This infers that
  \begin{align}\label{sd-eq5}
&\int_{\partial \Omega}(|D v |^2+\ez)^{ \gz+\beta }(D^2 v D v -\bdz  v D v )\cdot \nu\xi^2\,d\mathcal{H}\nonumber\\
&\quad
\le \int_{\partial \Omega}|\mathcal{B}|(|D v |^2+\ez)^{ \gz+\beta }|D v |^2\xi^2\,d\mathcal{H}^{n-1}.
\end{align}

 For the first term of the right-hand side in \eqref{sd-eq3}, a Young's inequality
 implies that
   \begin{align}\label{sd-eq6}
&-2\int_{\Omega}\xi(|D v |^2+\ez)^{ \gz+\beta  }(D^2 v D v -\bdz  v D v )
\cdot D\xi\,dx\nonumber \\
&\le\frac 14 C_1\int_{\Omega}|D^2 v |^2(|D v |^2+\ez)^{ \gz+\beta }
\xi^2\,dx+C\int_{\Omega}(|D v |^2+\ez)^{ \gz+ \bz +1}
|D\xi|^2\,dx.
\end{align}

Combining \eqref{sd-eq2}-\eqref{sd-eq6}, we conclude the desired \eqref{sd-eq7}.

\end{proof}

Now we prove Lemma \ref{sd-es}.
\begin{proof}
[Proof of Lemma \ref{sd-es}]

(i)
Multiplying both sides
 of the inequality \eqref{key-in} by $(|D v |^2+\ez)^{\bz}$ with $\bz\ge 0$
and  integrating   over $\Omega$, we obtain
 \begin{align}\label{sd-eq1}
&C_1
\int_{\Omega}|D^2 v |^2(|D v |^2+\ez)^{\gz+\beta }
 \,dx\nonumber \\
&\le C_2
\int_{\Omega} ((|Dv|^2+\ez)^{\frac \gz 2}\bdz^N_{p,\ez}v)^2 (|D v |^2+\ez)^{\bz}
 \,dx\nonumber\\
&\quad+
\int_{\Omega}{\rm div}\left((|D v |^2+\ez)^{\gz}(D^2 v D v -\bdz  v D v )\right)(|D v |^2+\ez)^{\bz} \,dx.
\end{align}
Via integration by parts for the second term of right-hand side we get
  \begin{align*}
&\int_{\Omega}{\rm div}\left((|D v |^2+\ez)^{\gz}(D^2 v D v -\bdz  v D v )\right) (|D v |^2+\ez)^{\bz} \,dx\nonumber \\
&=  \int_{\partial \Omega}(|D v |^2+\ez)^{\gz+\bz}(D^2 v D v -\bdz  v D v )\cdot \nu\,d\mathcal{H}^{n-1}\nonumber \\
&\quad -\beta \int_{\Omega}[|D^2vDv|^2-\bdz v\bdz_\fz v](|D v |^2+\ez)^{ \gz+\bz-1} \,dx.
\end{align*} Since  $\Omega$ is convex,  it is known that
$$(D^2 v D v -\bdz  v D v )
\cdot \nu\le0 \quad {\rm on}\quad \partial\Omega.$$
Observe that $\bdz^N_{p,\ez} v=\bdz v+(p-2)\bdz^N_{\fz,\ez}v$.
Then Young's inequality with any $\eta>0$ and $\bz\ge 0$ leads to
\begin{align*}
& \beta \int_{\Omega}[|D^2vDv|^2-\bdz v\bdz_\fz v](|D v |^2+\ez)^{\gz+\bz-1}\,dx\nonumber\\
&=-\beta \int_{\Omega}[|D^2vDv|^2+(p-2)\frac{(\bdz_\fz v)^2}{
|Dv|^2+\ez}](|D v |^2+\ez)^{\gz+\bz-1}\,dx\nonumber\\
&\quad-\beta \int_{\Omega}(|Du^{\ez}|^2+\ez)^{\frac \gz 2}\bdz^N_{p,\ez}u^{\ez}\bdz_\fz v(|D v |^2+\ez)^{\gz+\bz-1}\,dx\\
&\le
\eta\int_{\Omega}|
D^2v|^2(|D v |^2+\ez)^{\gz+\bz}\,dx\\
&\quad\quad+\eta^{-1}\beta ^2\int_{\Omega}((|Du^{\ez}|^2+\ez)^{\frac \gz 2}\bdz^N_{p,\ez}u^{\ez})^2(|D v |^2+\ez)^{\bz}\,dx.
\end{align*}
Let $\eta=\frac {C_1} 2$. We therefore obtain
 \begin{align*}
& \int_{\Omega}{\rm div}\left((|D v |^2+\ez)^{\gz}(D^2 v D v -\bdz  v D v )\right) (|D v |^2+\ez)^{\bz} \,dx\nonumber \\
&\le \frac{C_1}2
\int_{\Omega}|D^2 v |^2(|D v |^2+\ez)^{\gz+\bz}
 \,dx\nonumber \\
&\quad + C (1+\bz^2)
\int_{\Omega} ((|Dv|^2+\ez)^{\frac \gz 2}\bdz^N_{p,\ez}v)^2 (|D v |^2+\ez)^{\bz}
 \,dx.
\end{align*}
Combining together \eqref{sd-eq1} we get the desired result.

(ii) It suffices to bound the last term of right-hand side in
\eqref{sd-eq7}. By an argument of \cite{cm19}, we can
suppose that  $\xi\in C^\fz_c(B_r(x_0))$ for
 some  $x_0\in \partial \Omega$ and  $0<r<R_\Omega $. Here $R_{\Omega}$ is given by
Lemma \ref{trace-sob}.
It follows by Lemma \ref{trace-sob} that
\begin{align*}
 &  \int_{\partial\Omega }(|D v |^2+\ez)^{ \gz+\bz}|D v |^2|\mathcal{B}
|\xi^2\,d\mathcal{H}^{n-1} \\
&\le C\Psi_\Omega(r)\int_{ \Omega\cap B_r(x_0)}
|D[(|D v |^2+\ez)^{ \frac{\gz+\bz}2}D v\xi ]|^2\,dx
\end{align*}
for constant $C=C(L_{\Omega})$. Then we choose $$ \dz< \frac {C_1}{C4(\gz+\bz+1)^2}$$
such that $ \Psi_{\Omega}(r)\le \dz$ for some $0<r<R_{\Omega}$. Thus
 $$
 C_1 \Psi_\Omega(r)\le  \frac {C_1}{4(\gz+\bz+1)^2}.$$
This further leads to
\begin{align*}
 &  \int_{\partial\Omega }(|D v |^2+\ez)^{ \gz+\bz}|D v |^2|\mathcal{B}
|\xi^2\,d\mathcal{H}^{n-1} \\
&\le \frac {C_1}{4(\gz+\bz+1)^2}
\int_{\Omega}|D[(|D v |^2+\ez)^{\frac{ \gz+\beta}2 }D v]|^2
\xi^2\,dx+ \frac {C_1}{4(\gz+\bz+1)^2}\int_{\Omega}(|D v |^2+\ez)^{ \gz+\bz+1  }|D\xi|^2\,dx\\
&\le \frac {C_1}{4}
\int_{\Omega}|D^2v|^2(|D v |^2+\ez)^{ \gz+\beta }
\xi^2\,dx+ \frac {C_1}{4(\gz+1)^2}\int_{\Omega}(|D v |^2+\ez)^{ \gz+\bz+1  }|D\xi|^2\,dx.
\end{align*}
Plugging this in \eqref{sd-eq7} one gets

 \begin{align}\label{sd-eq8}
&\frac 14C_1
\int_{\Omega}|D^2 v |^2(|D v |^2+\ez)^{ \gz+\bz}
\xi^2\,dx\nonumber \\
&\le C (1+\bz^2)
\int_{\Omega} (\Delta_{p,\ez}^Nv)^2(|D v |^2+\ez)^{\gz+\bz}\xi^2\,dx+C \int_{\Omega}(|D v |^2+\ez)^{ \gz+\bz+1 }|D\xi|^2\,dx.
\end{align}

Next let $\{B_{r_k}\}_{1\le k\le  N}$ be a  covering of $\overline{\Omega}$ by balls $B_{r_k}$, with $r_{\Omega}/4\le r_k\leq r_{\Omega}$ for some $r_{\Omega}\in (0,1)$,
such that either $B_{r_k}$ is centered on $\partial\Omega$, or $B_{r_k} \Subset\Omega$.
Note that the covering can be chosen in the way that the multiplicity $N$ of overlapping among the balls $B_{r_k}$ depends only on $n$.
Let $\{\psi_k\}_{k\in N}$ be a family of functions such that $\psi_{k}\in C^\infty_c(B_{r_k})$ and $|D\psi_k|\le C_4(r_{\Omega})^{-1}$, and that
$\{\psi_k^2\}_{1\le k\le  N}$ is a partition of unity associated with the covering $\{B_{r_k}\}_{k\in N}$.  By applying inequality \eqref{sd-eq8}  with $\xi=\psi_k$ for each $k$, and summing the resulting inequalities, one obtains
\begin{align*}
&\int_{\Omega}|D^2 v |^2(|D v |^2+\ez)^{ \gz+\bz}\,dx\nonumber\\
&\le C(1+\bz^2)
\int_{\Omega} ((|Dv|^2+\ez)^{\frac\gz2}\Delta_{p,\ez}^Nv)^2(|Dv|^2+\ez)^{ \bz}\,dx+C\int_{\Omega}(|D v |^2+\ez)^{ \gz+\bz+ 1}
\,dx,
\end{align*}
where   $C=C(p,n,\gz,L_{\Omega},d_{\Omega})$. Now we use a Sobolev inequality as in
\cite[Lemma 7.2]{accfm} to get
\begin{align*}
C\int_{\Omega}(|D v |^2+\ez)^{ \gz+\bz+ 1}
&\le \frac {1} {16(\gz+1+\bz)^2}\int_{\Omega}|D(|D v |^2+\ez)^{ \frac{\gz+\bz+ 1}2}|^2\,dx\\
&\quad+C(\bz^2+1)
\left(\int_{\Omega}(|D v |^2+\ez)^{ \frac{\gz+\bz+ 1}2}
\,dx\right)^2\\
&\le \frac 12\int_{\Omega}|D^2 v |^2(|D v |^2+\ez)^{ \gz+\bz}\,dx\\
&\quad +C(\bz^2+1)
\left(\int_{\Omega}(|D v |^2+\ez)^{ \frac{\gz+\bz+ 1}2}
\,dx\right)^2
\end{align*}
for some constant $C=C(p,n,\gz,L_{\Omega},d_{\Omega})$. Hence we finish this proof.

\end{proof}

 \subsection {Proof of Lemma \ref{Lq-es}}
We first recall that  the following a priori estimate given by Theorem 3.2 in \cite{cm17}. This is very useful in proving Lemma \ref{Lq-es}. Note that the following result also holds
for the generalized solution; we refer to \cite{cm19}.

\begin{lem}\label{L1-f}
Let $p>1$ and $\ez>0$.
Suppose  that  $v\in C^2(\overline \Omega)$ is  a solution to
\begin{align*}
-{\rm div}((|D v |^2+\ez)^{\frac {p-2}{2}}D v )=f \ \mbox{in $ \Omega$; $v=0$ on $\partial\Omega$}
\end{align*}
for some $f\in C^{2}(\overline \Omega)$.
One has
\begin{align*}
\int_{\Omega}(|D v |^2+\ez)^{\frac {p-2}2}|D v |\,dx
\le C\int_{\Omega}|f|\,dx,
\end{align*}
where $C=C(p,n,L_\Omega,d_\Omega)$.
\end{lem}

We are ready to prove Lemma \ref{Lq-es}.
\begin{proof}[Proof of Lemma \ref{Lq-es}]
First, by a Poincar\'e-Sobolev inequality in \cite[Theomre 3.6]{cd24} we have
\begin{align*}
 \left(\int_{\Omega}
\big|
(|D v |^2+\ez)^{\frac{\gz+1}{2} }-\mint{-}_{\Omega}(|D v |^2+\ez)^{\frac{\gz+1}{2} }
\,dz\big|^{2^{\star}}\,dx
\right)^{\frac{1}{2^{\star}}}
&\le C
\left(\int_{\Omega}|D(|D v |^2+\ez)^{ \frac{\gz+1}2}|^2\,dx\right)^{\frac 12},
\end{align*}
where $C=C(n,d_{\Omega},L_{\Omega})$. This further leads to
\begin{align}\label{lq-eq1}
 \left(\int_{\Omega}
\big|
(|D v |^2+\ez)^{\frac{\gz+1}{2} 2^{\star}}\,dx
\right)^{\frac{1}{2^{\star}}}
&\le C
\left(\int_{\Omega} (|Dv|^2+\ez)^{\gz}|D^2u|^2\,dx\right)^{\frac 12}\nonumber\\
&\quad+C\int_{\Omega}|Dv|^{\gz+1}\,dx+C\ez^{\frac{\gz+1}2}.
\end{align}
Now it reduce to estimate the right-hand side of inequality \eqref{lq-eq1}.
 We will divide the following two cases.

(i) Assume that $\Omega$ is convex. For the first term of right-hand side in
\eqref{lq-eq1}, we apply \eqref{sd-es1} in Lemma \ref{sd-es} with $\bz=0$
to obtain
\begin{align}\label{lq-eq2}
\left(\int_{\Omega} (|Dv|^2+\ez)^{\gz}|D^2v|^2\,dx\right)^{\frac 12}
\le C\left(\int_{\Omega} (|Dv|^2+\ez)^{\gz}(\bdz^N_{p,\ez}v)^2\,dx\right)^{\frac 12}.
\end{align}

On the other hand, write
$$-{\rm div}((|D v |^2+\ez)^{\frac \gz 2}D v )
= (|Dv|^2+\ez)^{  \gz/2 }
 \bdz^N_{p,\ez}v +(p-2-\gz)(|D v |^2+\ez)^{\gz/2}\bdz^N_{\fz,\ez} v .
$$

Thus by Lemma \ref{L1-f} and H\"older inequality we have
\begin{align*}
&\int_{\Omega}(|D v |^2+\ez)^{\frac \gz 2}|D v | \,dx\\
&\le C\int_{\Omega}\big| (|Dv|^2+\ez)^{  \gz/2 }
 \bdz^N_{p,\ez}v +(p-2-\gz)(|D v |^2+\ez)^{\gz/2}\bdz^N_{\fz,\ez} v \big|\,dx \\
&\le C\int_{\Omega}  (|Dv|^2+\ez)^{  \gz/2 }|D^2v|  \,dx\\
&\le C\left(\int_{\Omega} (|Dv|^2+\ez)^{\gz}|D^2v|^2\,dx\right)^{\frac 12}.
\end{align*}
Noting that
\begin{align}\label{gz-eq1}
(|Dv|^2+\ez)^{\frac{\gz}2}|Dv|\ge |Dv|^{\gz+1}\quad \forall \gz>0
\end{align}
and
\begin{align}\label{gz-eq2}
|Dv|^{\gz+1}\le (|Dv|^2+\ez)^{\frac{\gz+1}2}&\le 2(|Dv|^2+\ez)^{\frac{\gz}2}|Dv|
+2(|Dv|^2+\ez)^{\frac{\gz}2}\ez^{\frac 12}\nonumber\\
&\le 2(|Dv|^2+\ez)^{\frac{\gz}2}|Dv|+2\ez^{\frac{\gz+1}2}\quad \forall -1<\gz<0.
\end{align}
Now combing \eqref{lq-eq1} and \eqref{lq-eq2} we deduce
$$\left(\int_{\Omega} (|Dv|^2+\ez)^{\gz}|D^2v|^2\,dx\right)^{\frac 12}+\int_{\Omega}|Dv|^{\gz+1}\,dx \le  C\left(\int_{\Omega} (|Dv|^2+\ez)^{\gz}(\bdz^N_{p,\ez}v)^2\,dx\right)^{\frac 12}
+C\ez^{\frac{\gz+1}2}.$$
This proves \eqref{lq-es1}.

(ii) Assume that $\Omega$ is not convex. Let $\lim_{r\to 0^+}\Psi_{\Omega}(r)\le \dz$ where
$\dz$ is given by Lemma \ref{Lq-es}.
Thanks to \eqref{lq-eq1} and
\eqref{sd-es2} with $\bz=0$ in Lemma \ref{Lq-es}, we only need to
bound the second term of right-hand side in \eqref{lq-eq1}. Namely,
it remains to estimate the upper bound of $\||Dv|^{\gz+1}\|_{L^1(\Omega)}$.
We also consider the following two subcases.

$\bullet$ Assume that $\gz>p-2$. We show that
\begin{align*}
\left(\int_{\Omega}|D v |^{\gz+1}\,dx\right)^2
\le C\left[\int_{\Omega} (|Dv|^2+\ez)^{  \gz }
(\bdz^N_{p,\ez}v)^2\,dx+
\left(\int_{\Omega}| v |^2\,dx\right)^{\gz+1}+\ez^{\gz+1}\right].
\end{align*}

Write  \begin{align*}
-{\rm div}((|D v |^2+\ez)^{\frac \gz 2}D v )
= (|Dv|^2+\ez)^{  \gz/2 }
 \bdz^N_{p,\ez}v +(p-2-\gz)(|D v |^2+\ez)^{\gz/2}\bdz^N_{\fz,\ez} v .
\end{align*}
Multiplying both sides by $ v $ and integrating
over on $\Omega$,  via integration by parts we obtain
\begin{align}\label{lq-eq3}
&\int_{\Omega}(|D v |^2+\ez)^{\frac \gz 2}|D v |^2\,dx\nonumber\\
&\quad= \int_{\Omega}-{\rm div}((|D v |^2+\ez)^{\frac \gz 2}D v ) v\,dx
\nonumber \\
&\quad=\int_{\Omega}v (|Dv|^2+\ez)^{  \gz/2 }
 \bdz^N_{p,\ez}v \,dx+(p-2-\gz)
\int_{\Omega} v (|D v |^2+\ez)^{\gz/2}\bdz^N_{\fz,\ez} v \,dx.
\end{align}
where  we used the fact that $ v =0$ on $\partial\Omega$.

 An application of H\"older's inequality  leads to
\begin{align*}
\int_{\Omega}v (|Dv|^2+\ez)^{  \gz/2 }
 \bdz^N_{p,\ez}v \,dx\le \|(|Dv|^2+\ez)^{  \gz/2 }
 \bdz^N_{p,\ez}v\|_{L^2(\Omega)}\| v \|_{L^2(\Omega)}.
\end{align*}
 Moreover, by H\"older's inequality and also Lemma \ref{sd-es}  we obtain
 \begin{align*}
&(p-2-\gz)
\int_{\Omega} v (|D v |^2+\ez)^{\gz/2}\bdz^N_{\fz,\ez} v \,dx\\
&\le
C(p,\gz)\left(\int_{\Omega}(|D v |^2+\ez)^{\gz}|D^2 v |^2\,dx\right)^{\frac 12}
\left(\int_{\Omega}| v |^2\,dx\right)^{\frac 12}\\
 &\le C   \| v \|_{L^2(\Omega)}\left\{ \|(|Dv|^2+\ez)^{  \gz/2 }
 \bdz^N_{p,\ez}v\|_{L^2(\Omega)}   +  \| |D v|^{\gz+1}\|_{L^1(\Omega)}+\ez^{\frac{\gz+1}2} \right\}.
\end{align*}

Thus it follows by \eqref{lq-eq3} that
\begin{align*}
&\int_{\Omega} (|D v |^2+\ez)^{\frac \gz 2}|D v |^2\,dx\\
&\le C\| v \|_{L^2(\Omega)} \left\{  \|(|Dv|^2+\ez)^{  \gz/2 }
 \bdz^N_{p,\ez}v\|_{L^2(\Omega)}+ \| |D v|^{\gz+1}\|_{L^1(\Omega)}+\ez^{\frac{\gz+1}2}\right\}.
\end{align*}

By \eqref{gz-eq1}-\eqref{gz-eq2}, from above we conclude that
\begin{align*}
\int_{\Omega} |Dv|^{\gz+2}\,dx
\le C\| v \|_{L^2(\Omega)} \left\{  \|(|Dv|^2+\ez)^{  \gz/2 }
 \bdz^N_{p,\ez}v\|_{L^2(\Omega)}+ \| |D v|^{\gz+1}\|_{L^1(\Omega)}+\ez^{\frac{\gz+1}2}\right\}+\ez^{\frac{\gz+2}2}.
\end{align*}
From this, by H\"older's inequality and
 Young's inequality, one gets
\begin{align*}
&\left(\int_{\Omega}|D v |^{\gz+1}\,dx\right)^2
 \\
&\le C\left(\int_{\Omega}|D v |^{\gz+2}\,dx\right)
^{\frac{2(\gz+1)}{\gz+2}}\nonumber\\
&\le C[ \| v \|_{L^2(\Omega)} ^{2(\gz+1)}]^{\frac1{\gz+2}}  \left\{  \|(|Dv|^2+\ez)^{  \gz/2 }
 \bdz^N_{p,\ez}v\|_{L^2(\Omega)}^2+ \| |D v|^{\gz+1}\|^2_{L^1(\Omega)}+\ez^{{\gz+1} } \right\}^{\frac { \gz+1 }{\gz+2}}+\ez^{ {\gz+1} }\nonumber\\
 &\le\frac12\{ \|(|Dv|^2+\ez)^{  \gz/2 }
 \bdz^N_{p,\ez}v\|^2_{L^2(\Omega)}+  \| |D v|^{\gz+1} \|^2_{L^1(\Omega)} +\ez^{ {\gz+1} }\} +C\| v \|^{2(\gz+1)}_{L^2(\Omega)} +\ez^{ {\gz+1} }.
\end{align*}
Absorbing the term $ \frac12 \| |D v|^{\gz+1} \|^2_{L^1(\Omega)}$ in right-hand side to the left-hand side,
 one has
\begin{align*}
\left(\int_{\Omega}|D v |^{\gz+1}\,dx\right)^2
\le C\left[\int_{\Omega} (|Dv|^2+\ez)^{  \gz }
(\bdz^N_{p,\ez}v)^2\,dx+
\left(\int_{\Omega}| v |^2\,dx\right)^{\gz+1}+\ez^{\gz+1}\right].
\end{align*}
This proves \eqref{lq-es2}.

$\bullet$ Assume that $\gz\le p-2$. Recalling that \eqref{lq-es1} gives us
\begin{align*}
 \left(\int_{\Omega}(|D v |^2+\ez)^{\frac{\gz+1}{2}2^{\star}}
\right)^{\frac{1}{2^{\star}}}
& \le C\left(\int_{\Omega} (|Dv|^2+\ez)^{  \gz }
|D^2v|^2\,dx\right)^{\frac 12}
+ C\int_{\Omega} |D v |^ {\gz+1} \,dx   +C\ez^{\frac{\gz+1}2}.
\end{align*}
By Lemma \ref{sd-es}, one has
 \begin{align*}
&\left(\int_{\Omega}|Dv|^{2^{\star}(\gz+1)}\,dx\right)^{\frac 1{2^{\star}(\gz+1) }}\\
&\quad
\le \bar C\left[\left(\int_{\Omega} (|Dv|^2+\ez)^{  \gz }
 (\bdz^N_{p,\ez}v)^2\,dx\right)^{\frac1{2(\gz+1)}}+
 \left(\int_{\Omega}|Dv|^{\gz+1}\,dx\right)^{\frac1{\gz+1}} +\ez^{\frac12}  \right].
\end{align*}
To get \eqref{lq-es3}, it suffices to show
\begin{align} \label{lq-eqxx}
 &\bar C\left(\int_{\Omega}|D v |^{\gz+1}\,dx\right)^{\frac 1{\gz+1}}\nonumber\\
 &\le \frac 12\left(\int_{\Omega}|D v |^{(\gz+1)2^\star}
\right)^{\frac1{2^\star(\gz+1)}}+C\left(\int_{\Omega} |[|D v |^2+\ez]^{\gamma}(\Delta^N_{p,\ez})^2\,dx
\right)^{\frac 1{2(\gz+1)}}+C\ez^{\frac 12}.
\end{align}

 To this end, write
\begin{align*}
-{\rm div}((|D v |^2+\ez)^{\frac {p-2}{2}}D v )=-(|D v |^2+\ez)^{\frac {p-2 }{2}}\Delta^N_{p,\ez}v.
\end{align*}
Since  $ v =0$ on $\partial\Omega$,  we apply Lemma \ref{L1-f} to get
\begin{align*}
\int_{\Omega}(|D v |^2+\ez)^{\frac {p-2}2}|D v |\,dx
\le C\int_{\Omega}|(|D v |^2+\ez)^{\frac {p-2 }{2}}\Delta^N_{p,\ez} v|\,dx,
\end{align*}
where $C=C(p,n,\Omega)$.
By $ \gz<p-2$,   H\"older inequality and   $$ |D v|^{p-1}\le 2 (|D v |^2+\ez)^{\frac {p-2}2}|D v|+2\ez^{\frac{p-1}2},$$  one has
\begin{align*}
\bar C\left(\int_{\Omega}|D v |^{\gz+1}\,dx\right)^{\frac 1{\gz+1}}
&d\le C\left(\int_{\Omega}|D v |^{p-1}\,dx\right)^{\frac 1{p-1}}\nonumber\\
&\quad\le C\left(\int_{\Omega}[|D v |^2+\ez]^{\frac {p-2 }{2}}\Delta^N_{p,\ez} v|\,dx\right)^{\frac 1{p-1}}+C\ez^{\frac {1}2}.
\end{align*}
By  H\"older inequality again,
\begin{align*}
 &\bar C\left(\int_{\Omega}|D v |^{\gz+1}\,dx\right)^{\frac 1{\gz+1}}\nonumber\\
&\quad\le C\left(\int_{\Omega}[|D v |^2+\ez]^{\gamma}(\Delta^N_{p,\ez} v)^2\,dx\right)^{\frac 1{2(p-1)} }
\left[\left(\int_{\Omega}|D v |^{2(p-2-\gz)}\,dx\right)^{\frac 1{2(p-1)}}
+\ez^{\frac {p-2-\gz}{2(p-1)}}\right]+C\ez^{\frac {1}2}.
\end{align*}

Since
$$p-2> \gz>-1+(p-1)\frac{n-2}{2(n-1)},$$
then
 $$0< 2(p-2-\gz)\le (\gz+1)2^\star.$$
Indeed, in the case $n=2$, $2^\star$ denotes  any $ q\ge \frac{2(p-2-\gz)}{1+\gz}$,
and hence
$$2(p-2-\gz)= (1+\gz)\frac{2(p-2-\gz)}{1+\gz} \le (1+\gz)2^\star.$$
If $n\ge 3$,  by $2^\star=\frac {2n}{n-2}$,  a direction calculation gives
\begin{align*}
(\gz+1)\frac{2n}{n-2}-2(p-2-\gz)
&=\gz\left(\frac{2n}{n-2}+2\right)+\frac{2n}{n-2}-2(p-2)\\
&=\gz \frac{2n+2(n-2)}{n-2}+\frac{2n+2(n-2)}{n-2}-2(p-1)\\
&=\frac{4(n-1)}{n-2}
\left[\gz-\left(-1+(p-1)\frac{n-2}{2(n-1)}\right)\right]>0.
\end{align*}

Since  $ \frac{2^\star (\gz+1)}{ 2(p-2-\gz) }\ge 1$,
applying H\"older inequality  we have
\begin{align*}
\left(\int_{\Omega}|D v |^{2(p-2-\gz)}\,dx\right)^{\frac 1{2(p-1)}}
\le C\left(\int_{\Omega}|Du|^{(\gz+1)2^\star}\,dx\right)^{\frac{1}{2^\star (\gz+1)}
  \frac{p-2-\gz}{p-1}}.
\end{align*}
Thus
\begin{align*}
 &\bar C\left(\int_{\Omega}|D v |^{\gz+1}\,dx\right)^{\frac 1{\gz+1}}\nonumber\\
&\quad\le C\left(\int_{\Omega}[|D v |^2+\ez]^{\gamma}(\Delta^N_{p,\ez})^2\,dx\right)^{\frac 1{2(p-1)} }
\left[\left(\int_{\Omega}|Du|^{(\gz+1)2^\star}\,dx\right)^{\frac{1}{2^\star (\gz+1)}
  \frac{p-2-\gz}{p-1}}
+\ez^{\frac {p-2-\gz}{2(p-1)}}\right]+C\ez^{\frac {1}2}.
\end{align*}
Noting that $ 1<\frac{p-1}{p-2-\gz}<\fz$ and its conjugate exponent is
$ \frac{p-1}{1+\gz} $, we then employ Young's inequality to
get \eqref{lq-eqxx} as desired.

\end{proof}

\section{Some   H\"older  estimates}

In this section, we assume that the following:
\begin{itemize}
\item[$\bullet$] $\Omega$ is smooth bounded domain
in $\rr^n$ with $n\ge 2$.
\item[$\bullet$] $v\in C^\fz(\overline \Omega)$ and
$v=0$ on $\partial \Omega$.
\item[$\bullet$] Let $\ez\in (0,1]$. Let $p>1$ and let $\gz>-1$ satisfy
\eqref{pgz}.
\end{itemize}
We will focus on the H\"older estimate. Recall that the norm of
H\"older space $C^{0,\az}(\overline \Omega)$ with $\az\in (0,1)$ is defined by
$$\|w\|_{C^{0,\az}(\overline \Omega)}=\max_{\overline \Omega}|w|+\sup_{x\neq y\in\overline\Omega}\frac{|w(x)-w(y)|}{|x-y|^\alpha}<\fz$$
for all $w\in C^{0,\az}(\overline \Omega)$.

In the case,  $2^{\star} (\gz+1)>n$, from Lemma \ref{Lq-es} and Morrey's inequality for Sobolev space, we  conclude the following.

\begin{lem}\label{holder-1}
Let $2^{\star} (\gz+1)>n$ and 
write $\az=1-\frac{n}{2^{\star} (\gz+1)}$.
\begin{itemize}

\item[$\bullet$] If $\Omega$ is convex, then
$$\|v\|_{C^{0,\az}(\overline \Omega)}\le C\|(|Dv|^2+\ez)^{ \gz/2 }\bdz^N_{p,\ez}v\|^{\frac 1{\gz+1}}_{L^2(\Omega)}
+C\ez^{\frac 12}.$$

\item[$\bullet$]  Under the same conditions as in Lemma
\ref{Lq-es} (ii).  If $\gz\le p-2$, then
$$\|v\|_{C^{0,\az}(\overline \Omega)}\le C\|(|Dv|^2+\ez)^{  \gz/2 }\bdz^N_{p,\ez}v\|^{\frac 1{\gz+1}}_{L^2(\Omega)}+C\ez^{\frac 12}.$$

\item[$\bullet$]  Under the same conditions as in Lemma
\ref{Lq-es} (ii).  If $\gz>p-2$, then
$$\|v\|_{C^{0,\az}(\overline \Omega)}\le C\left[\|(|Dv|^2+\ez)^{  \gz/2 }\bdz^N_{p,\ez}v\|^{\frac 1{\gz+1}}_{L^2(\Omega)}
+\|v\|_{L^2(\Omega)}\right]+C\ez^{\frac 12}.$$
\end{itemize}

All above constant $C$ only depend on $p$, $n$,
$\gz$, $L_{\Omega}$ and $d_{\Omega}$.
\end{lem}

\begin{rem}\rm
Suppose that
$$(|Dv|^2+\ez)^{ \gz/2 }\bdz^N_{p,\ez}v\to
|Dv|^{\gz}\bdz^N_p v\quad{\rm in}\quad L^2(\Omega)
\quad{\rm as}\quad \ez\to 0.$$
Lemma \ref{holder-1} tells us that  the $C^{0,\az}$-norm of $v$ can be bounded
by the $L^2$-norm of $|Dv|^{\gz}\bdz^N_p v$ if $\Omega$ is convex and $2^{\star}(\gz+1)>n$.
However, if $2^{\star}(\gz+1)<n$ with $n\ge 3$ and $2^{\star}(\gz+1)=n$ with $n\ge 4$, this is not correct; we refer
to an counterexample in Proposition  \ref{example} at end of this section.

\end{rem}

In the case $2^{\star}(\gz+1)\le n$, via Moser's iteration  we also deduce the following
the H\"older estimate, where we use
$\|(|Dv|^2+\ez)^{  \gz/2 }\bdz^N_{p,\ez}v\| _{L^\infty(\Omega)}$ instead of
$\|(|Dv|^2+\ez)^{  \gz/2 }\bdz^N_{p,\ez}v\| _{L^2(\Omega)}$  in Lemma \ref{holder-1}.

\begin{lem}\label{holder-2}
Let $2^{\star} (\gz+1)\le n$ and let $k_0$ be a positive integer so that
$k_0>\ln \frac{n}{2^{\star} (\gz+1)}$. Write  $\az_0=1-\frac{n}{(2^{\star})^{k_0+1}(\gz+1)}$.

\begin{itemize}

\item[$\bullet$] If $\Omega$ is convex, then
$$\|v\|_{C^{0,\az_0}(\overline \Omega)}\le C\|(|Dv|^2+\ez)^{ \gz/2 }\bdz^N_{p,\ez}v\|^{\frac 1{\gz+1}}_{L^{\fz}(\Omega)}
+C\ez^{\frac 12}.$$

\item[$\bullet$]  Under the same conditions as in Lemma
\ref{Lq-es} (ii).  If $\gz>p-2$, then
$$\|v\|_{C^{0,\az_0}(\overline \Omega)}\le C\left[\|(|Dv|^2+\ez)^{  \gz/2 }\bdz^N_{p,\ez}v\|^{\frac 1{\gz+1}}_{L^{\fz}(\Omega)}
+\|v\|_{L^2(\Omega)}\right]+C\ez^{\frac 12}.$$

\item[$\bullet$]  Under the same conditions as in Lemma
\ref{Lq-es} (ii).  If $\gz\le p-2$, then
$$\|v\|_{C^{0,\az_0}(\overline \Omega)}\le C\|(|Dv|^2+\ez)^{  \gz/2 }\bdz^N_{p,\ez}v\|^{\frac 1{\gz+1}}_{L^{\fz}(\Omega)}+C\ez^{\frac 12}.$$
\end{itemize}

All above constants $C$ only depend on $k_0$, $p$, $n$,
$\gz$, $L_{\Omega}$ and $d_{\Omega}$.

\end{lem}

The following shows that the
$\|(|Dv|^2+\ez)^{  \gz/2 }\bdz^N_{p,\ez}v\| _{L^\infty(\Omega)}$ above cannot be replaced by
$\|(|Dv|^2+\ez)^{  \gz/2 }\bdz^N_{p,\ez}v\| _{L^2(\Omega)}$.
  Observe that
$2^{\star} (\gz+1)\le n$ if and only if $\gz\le \frac{n-4}2$ whenever $n\ge 3$.

\begin{prop}\label{example}
Assume that $-1<\gz<\frac{n-4}2$ with $n\ge 3$ and $\gz=\frac{n-4}2$
with $n\ge 4$. There exists $w_{k}\in C^\fz(\overline B_1)$ with
$w_{k}=0$ on $\partial B_1$ such that
$$\||Dw_k|^{\gz}\bdz^N_{p}w_{k}\|_{L^2(B_1)}\le C(n,p,\gz)
\quad{\rm while}\quad \|w_{k}\|_{L^\fz(B_1)}\to +\fz
\quad{\rm as}\quad k\to \fz.$$
\end{prop}

We also establish a local version for H\"older estimate.
\begin{lem}\label{local-hol}
Assume that $v\in C^\fz(\Omega)$ with
$v=0$ on $\partial\Omega$.
There exists $\az=\az(n,\gz)\in (0,1)$ such that the following holds:



\begin{itemize}
\item[(i)] If $\Omega$ convex or $\Omega$ is not convex whenever $\gz<p-2$, then
\begin{align*}
\|v\|_{C^{0,\az}(B_{r/2}(y))}\le C\|(|Dv|^2+\ez)^{\gz/2}\Delta_{p,\ez}^Nv\|
^{1/(\gz+1)}_{L^\fz(B_r(y))}+C\|(|Dv|^2+\ez)^{  \gz/2 }\bdz^N_{p,\ez}v\|^{\frac 1{\gz+1}}_{L^2(\Omega)}+C\ez^{\frac 12},
\end{align*}
where $C=C(p,n,\gz,r,\Omega)$.

\item[(ii)] If $\Omega$ is not convex whenever $\gz>p-2$, then
\begin{align*}
&\|v\|_{C^{0,\az}(B_{r/2}(y))}\\
&\le C\|(|Dv|^2+\ez)^{\gz/2}\Delta_{p,\ez}^Nv\|
^{1/(\gz+1)}_{L^\fz(B_r(y))}+C\|(|Dv|^2+\ez)^{  \gz/2 }\bdz^N_{p,\ez}v\|^{\frac 1{\gz+1}}_{L^2(\Omega)}+C\|v\|_{L^2(\Omega)}+C\ez^{\frac 12},
\end{align*}
where $C=C(p,n,\gz,r,\Omega)$.
\end{itemize}
\end{lem}

\begin{rem}\rm
(i) In dimension $n\ge 4$, thanks to Proposition \ref{example}, we know that the H\"older estimate in
Lemma \ref{holder-1}  is sharp for the convex domain $\Omega$. Namely, if
$\gz\le \frac{n-4}2$, we can not control the H\"older norm of
$v$ by using the $L^2$-norm of $(|Dv|^2+\ez)^{\gz/2}\bdz^N_{p,\ez} v$ uniformly
in $v$ and $\ez$.

(ii) For the case $\gz=\frac {n-4}2$ with $n=3$,
it is unclear whether Proposition
\ref{example} holds.

(iii) In Lemma \ref{local-hol} (ii), when $\Omega$ is not convex and $\gz>p-2$, it remains open to control term
$\|v\|_{L^2(\Omega)}$ via $\|(|Dv|^2+\ez)^{  \gz/2 }\bdz^N_{p,\ez}v\|_{L^\fz_{\loc}(\Omega)}$ and
$\|(|Dv|^2+\ez)^{  \gz/2 }\bdz^N_{p,\ez}v\|_{L^2(\Omega)}$.
\end{rem}

\subsection{Proof of   Lemma \ref{holder-2}}
We now prove Lemma \ref{holder-2}.

\begin{proof}[Proof of Lemma \ref{holder-2}]
Denote by the $q_k=(\gz+1)\left(2^{\star}\right)^k$ with $k\ge 0$. Write
$$M=\|(|Dv|^2+\ez)^{\gz/2}\bdz^N_{p,\ez}v\|_{L^\fz(\Omega)}.$$
We first claim
that there exists a constant $C_0=C_0(p,n,\gz,L_{\Omega},d_{\Omega})>2$ such
that
\begin{align}\label{cla-hi}
& \left(\int_{\Omega}
|\sqrt{|Dv|^2+\ez}|^{q_{k+1}}
\right)^{\frac{1}{q_{k+1}}}\nonumber\\
  & \le C_0M^{\frac{1}{\gz+1}}(2^{\star})^k
+(C_0)^{\frac 1{\gz+1}\frac{k+1}{(2^{\star})^{k+1}}}\left(\int_{\Omega}|\sqrt{|Dv|^2+\ez}|^{q_k}\,dx\right)^{\frac 1{q_k}}
\end{align}
for all $k\ge 0$.
Assume that this holds for the moment.  Since $\sum_{j=0}^{\fz}\frac{k}{(2^{\star})^k}\le C(n)$,
then by an iteration we conclude that
\begin{align}\label{lq-it}
& \left(\int_{\Omega}
|\sqrt{|Dv|^2+\ez}|^{q_{k+1}}
\right)^{\frac{1}{q_{k+1}}}\nonumber\\
  & \le CM^{\frac{1}{\gz+1}}\sum_{j=0}^k(2^{\star})^j
+C
\left(\int_{\Omega}|\sqrt{|Dv|^2+\ez}|^{\gz+1}\,dx\right)^{\frac 1{\gz+1}}.
\end{align}
From this and by Morrey's inequality with $q_{k+1}>n$ for some $k>\ln \frac{n}{2^{\star}(\gz+1)}$,
then we only need to bound
the second term of right-hand side in \eqref{lq-it}. If $\Omega$ is convex,
then  it follows from \eqref{lq-es1} in Lemma \ref{Lq-es}.
Suppose now that $\Omega$ is not convex and that
$\lim_{r\to 0^{+}}\Psi_{\Omega}(r)<\dz$ for some $\dz>0$.
Similarly, this also holds via \eqref{lq-es2} whenever $\gz\le p-2$ and
by \eqref{lq-es3} whenever $\gz>p-2$.

In what follows, we prove claim \eqref{cla-hi}. Observe that $\Omega$ is Lipschitz domain. By Theorem 3.6 in \cite{cd24},
we can use a Poincar\'e-Sobolev inequality to get
\begin{align*}
 \left(\int_{\Omega}
\big|
(|D v |^2+\ez)^{\frac{\gz+1+\bz}{2} }-\mint{-}_{\Omega}(|D v |^2+\ez)^{\frac{\gz+1+\bz}{2} }
\,dz\big|^{2^{\star}}\,dx
\right)^{\frac{1}{2^{\star}}}
&\le C
\left(\int_{\Omega}|D(|D v |^2+\ez)^{ \frac{\gz+1+\bz}2}|^2\,dx\right)^{\frac 12},
\end{align*}
where $C=C(n,d_{\Omega},L_{\Omega})$. Thus via a triangle inequality one gets
\begin{align*}
 &\left(\int_{\Omega}
(|D v |^2+\ez)^{\frac{\gz+1+\bz}{2}2^{\star} }\,dx
\right)^{\frac{1}{2^{\star}}}\\
&\le C
\left(\int_{\Omega}|D(|D v |^2+\ez)^{ \frac{\gz+1+\bz}2}|^2\,dx\right)^{\frac 12}
+C\int_{\Omega}(|D v |^2+\ez)^{\frac{\gz+1+\bz}2}\,dx.
\end{align*}
Applying Lemma \ref{sd-es} for the second term of right-hand side we further have
  \begin{align*}
& \left(\int_{\Omega}(|D v |^2+\ez)^{\frac{\gz+1+\bz}{2}2^{\star}}
\right)^{\frac{1}{2^{\star}}}\\
&\le  C(1+\bz^2)\left(\int_{\Omega} (|Dv|^2+\ez)^{  \gz+\bz }
(\Delta^N_{p,\ez}v)^2\,dx\right)^{\frac 12}+C(1+\bz^2)\int_{\Omega}(|D v |^2+\ez)^{\frac{\gz+1+\bz}2}\,dx,
\end{align*}
where $C=C(d_{\Omega},L_{\Omega},p,n,\gz)$. If $\bz=0$, then \eqref{cla-hi} clearly follows with
$k=0$.
Assume that $\bz>0$ and $k\ge 1$.
Now we apply H\"older inequality with exponents
$\frac{2^{\star}(\gz+1+\bz)}{2\bz}$ and $\frac{2^{\star}(\gz+1+\bz)}{2^{\star}(\gz+1+\bz)-2\bz}$
to obtain
\begin{align*}
&C(1+\bz^2)\left(\int_{\Omega} (|Dv|^2+\ez)^{  \gz+\bz }
(\Delta^N_{p,\ez}v)^2\,dx\right)^{\frac 12}\\
&\le C(1+\bz^2)M|\Omega|^{\frac{2^{\star}(\gz+1+\bz)-2\bz}{2^{\star}(\gz+1+\bz)}}\left(
\int_{\Omega}(|Dv|^2+\ez)^{\frac{2^{\star}(\gz+1+\bz)}{2} }\,dx \right)
^{\frac{\bz}{2^{\star}(\gz+1+\bz)}}
\end{align*}
Then using a weighted Young inequality with exponents
$\frac{\gz+1+\bz}{\bz}$ and $\frac{\gz+1+\bz}{\gz+1}$ we further have
\begin{align*}
&C(1+\bz^2)M|\Omega|^{\frac{2^{\star}(\gz+1+\bz)-2\bz}{2^{\star}(\gz+1+\bz)}}\left(
\int_{\Omega}(|Dv|^2+\ez)^{\frac{2^{\star}(\gz+1+\bz)}{2} }\,dx \right)
^{\frac{\bz}{2^{\star}(\gz+1+\bz)}}\\
&\le \frac 12\left(\int_{\Omega}(|D v |^2+\ez)^{\frac{\gz+1+\bz}{2}2^{\star}}
\right)^{\frac{1}{2^{\star}}}
+\frac{\gz+1}{\gz+1+\bz}
\left(\frac{2\bz}{\gz+1+\bz}\right)^{\frac{\bz}{\gz+1}}[CM(1+\bz^2)]^{\frac{\gz+1+\bz}{\gz+1}}
|\Omega|^{\frac{2^{\star}(\gz+1+\bz)-2\bz}
{2^{\star}(\gz+1)}}
\end{align*}

Now we set
$$\bz_k=(\gz+1)\left(2^{\star}\right)^k-(\gz+1)$$
for all $k\ge 1$. Note that $q_k=(\gz+1)\left(2^{\star}\right)^k$.
Then we deduce
\begin{align*}
& \left(\int_{\Omega}
|\sqrt{|Dv|^2+\ez}|^{q_{k+1}}
\right)^{\frac{1}{q_{k+1}}}\\
  & \le C_0(2^{\star})^k M^{\frac{1}{\gz+1}}
+(C_0)^{\frac 1{\gz+1}\frac{k+1}{(2^{\star})^{k+1}}}\left(\int_{\Omega}|\sqrt{|Dv|^2+\ez}|^{q_k}\,dx\right)^{\frac 1{q_k}}
\end{align*}
for some $C_0=C_0(p,n,\gz,L_\Omega,d_{\Omega})>2$ and
for all $k\ge 1$.  Hence we finish this proof.
\end{proof}

%
%
%
%
%
%

The proof Lemma \ref{local-hol} is much similar to Lemma \ref{holder-2}.

\begin{proof}[Proof of Lemma \ref{local-hol}]
Without loss of generality, we assume that   $y=0$ and $B_r\Subset \Omega$.
Set
$$M_r=\|(|Dv|^2+\ez)^{\gz/2}\Delta_{p,\ez}^Nu\|_{L^\fz(B_r)}.$$

Observe that \eqref{sd-eq7} holds.
Choose a cut-off function $\xi\in C^\fz_c(B_r)$ with $0\le \xi\le1$ in \eqref{sd-eq7} to obtain
 \begin{align}
&\frac 12C_1
\int_{B_r}|D^2 v |^2(|D v |^2+\ez)^{ \gz+\bz}
\xi^2\,dx\nonumber \\
&\le C (1+\bz^2)
\int_{B_r} (|Dv|^2+\ez)^{\gz+\bz}(\Delta_{p,\ez}^Nu)^2\xi^2\,dx+C \int_{B_r}(|D v |^2+\ez)^{ \gz+\bz+1 }|D\xi|^2\,dx\nonumber
\\
&\le C (1+\bz^2) M_r^2
\int_{B_r} (|Dv|^2+\ez)^{\bz}\xi^2\,dx+C \int_{B_r}(|D v |^2+\ez)^{ \gz+\bz+1 }|D\xi|^2\,dx\nonumber
\end{align}
for all $\bz\ge 0$. A Sobolev inequality yields that
\begin{align}
 &\left(\int_{B_r}
\big|(|D v |^2+\ez)^{\frac{\gz+1+\bz}{2}}\xi\big|^{2^{\star}}\,dx
\right)^{\frac{2}{2^{\star}}}\nonumber\\
&\le C\int_{B_r}|D[(|D v |^2+\ez)^{ \frac{\gz+1+\bz}2}\xi]|^2\,dx\nonumber\\
&\le C (1+\bz^2)M_r^2
\int_{B_r} (|Dv|^2+\ez)^{ \bz} \xi^2\,dx+C \int_{B_r}(|D v |^2+\ez)^{ \gz+\bz+1 }|D\xi|^2\,dx\label{xx1}
\end{align}

We first assume that $\bz>0$.  
By H\"older's inequality and Young's inequality, we have
\begin{align*}
 & C (1+\bz^2)M^2_r
\int_{B_r} (|Dv|^2+\ez)^{\bz}\xi^2\,dx\\
&\le C (1+\bz^2)M_r|B_r|^{\frac{2^{\star}(\gz+1+\bz)-2\bz}{2^{\star}(\gz+1+\bz)}}
 \left(\int_{B_r} (|Dv|^2+\ez)^{\frac{2^{\star}}{2}(\gz+1+\bz)}\xi^{2^{\star}{\frac{\gz+1+\bz}{\bz}}}\,dx\right)^{\frac{2\bz}{2^{\star}(\gz+1+\bz)}}\\
&\le \frac 12\left(\int_{B_r} (|Dv|^2+\ez)^{\frac{2^{\star}}{2}(\gz+1+\bz)}\xi^{2^{\star}{\frac{\gz+1+\bz}{\bz}}}
\,dx\right)^{\frac{2}{2^{\star}}}\\
&\quad+\frac{\gz+1}{\gz+1+\bz}
\left(\frac{2\bz}{\gz+1+\bz}\right)^{\frac{\bz}{\gz+1}}[CM^2_r(1+\bz^2)]^{\frac{\gz+1+\bz}{\gz+1}}
|B_r|^{\frac{2^{\star}(\gz+1+\bz)-2\bz}
{2^{\star}(\gz+1)}}
\end{align*}
Since $0\le \xi\le 1$, one has $\xi^{2^{\star}{\frac{\gz+1+\bz}{\bz}}}
\le \xi^{2^{\star}}$. From this and \eqref{xx1}, we conclude that
\begin{align*}
 &\left(\int_{B_r}
\big|(|D v |^2+\ez)^{\frac{\gz+1+\bz}{2}}\xi\big|^{2^{\star}}\,dx
\right)^{\frac{2}{2^{\star}}}\\
&\le C\int_{B_r}(|D v |^2+\ez)^{ (\gz+\bz+1)}|D\xi|^{2}\,dx\\
&\quad+\frac{\gz+1}{\gz+1+\bz}
\left(\frac{2\bz}{\gz+1+\bz}\right)^{\frac{\bz}{\gz+1}}[CM^2_r(1+\bz^2)]^{\frac{\gz+1+\bz}{\gz+1}}
|B_r|^{\frac{2^{\star}(\gz+1+\bz)-2\bz}
{2^{\star}(\gz+1)}}
\end{align*}
Write
$$\bz_k=\left(\frac {2^{\star}}2\right)^{k}(\gz+1)-(\gz+1) $$
and
$$q_k=\left(\frac {2^{\star}}2\right)^{k}(\gz+1).$$
 We  derive that
\begin{align*}
& \left(\int_{B_r}
(|Dv|^2+\ez)^{q_{k+1}}\xi^{2^{\star}}
\right)^{\frac{1}{q_{k+1}}}\\
  & \le C_0(2^{\star})^k M_r^{\frac{2}{\gz+1}}
+(C_0)^{\frac 1{\gz+1}\frac{k+1}{(2^{\star}/2)^{k+1}}}\left(\int_{B_r}(|Dv|^2+\ez)^{q_k}
|D\xi|^2\,dx\right)^{\frac 1{q_k}}
\end{align*}
for some $C_0=C_0(p,n,\gz,r)>2$ and for all $k\ge 1$. Note that
$k=0$ also holds for this inequality.

For each $k\ge 0$, we choose $\xi\in C^{\fz}_c(B_{r_k})$ so that
 $$\xi=1\quad{\rm on}\ B_{r_{k+1}},\quad |D\xi|\le C2^k $$
where
$$r_k=\frac r2+\frac{r}{2^{k+1}}.$$
We then obtain
\begin{align*}
& \left(\int_{B_{r_{k+1}}}
(|Dv|^2+\ez)^{q_{k+1}}
\right)^{\frac{1}{q_{k+1}}}\\
  & \le C_0(2^{\star})^k M_r^{\frac{2}{\gz+1}}
+(C_0)^{\frac 1{\gz+1}\frac{k+1}{(2^{\star}/2)^{k+1}}}\left(\int_{B_{r_k}}(|Dv|^2+\ez)^{q_k}
\,dx\right)^{\frac 1{q_k}}
\end{align*}
By an iteration, one can find a $k_0=k_0(n,\gz)$ such that
$\left(\frac {2^{\star}}2\right)^{k_0+1}(\gz+1)>n$ and
\begin{align*}
& \left(\int_{B_{r_{k_0+1}}}
(|Dv|^2+\ez)^{q_{k_0+1}}
\right)^{\frac{1}{q_{k_0+1}}} \le C M_r^{\frac{2}{\gz+1}}
+C\left(\int_{B_{r}}(|Dv|^2+\ez)^{\gz+1}
\,dx\right)^{\frac 1{\gz+1}}
\end{align*}
If $\Omega$ is convex or $\Omega$ is not convex whenever $\gz\le p-2$, using a Morrey's inequality and Lemma \ref{Lq-es} we conclude this lemma.

\end{proof}
\subsection{Proof of Proposition \ref{example}}
We prove Proposition \ref{example} in the following.

\begin{proof}[Proof of Proposition \ref{example}]
Let  $\phi(s)\in C^\fz(0,+\fz)$ a smooth cut-off function such that
\begin{align}\label{hj-1}
\phi(s):=\left\{
\begin{aligned}
1&\quad{\rm if}\ &s\ge 2,\\
0&\quad {\rm if}\ &0\le s\le 1,
\end{aligned}
\right.
\end{align}
and
\begin{align}\label{hj-2}
|\phi'(s)|\le 2\quad \forall s\ge 0.
\end{align}

Given $
 \ez\in(0,1/2)$, set

$$
v^{\ez}(|x|)=\int^1_{|x|} s^{-1}\phi\left(\frac s \ez\right)\,ds\quad
\mbox{for\  all $|x|<1$.}$$

 For $x\in B_\ez$, thanks to  $\phi(\ez^{-1}s)=0$ when $ s\le \ez $,  we have
\begin{align}\label{cons}
v^{\ez}(x)=\int^\ez_{|x|} s^{-1}\phi\left(\frac s \ez\right)\,ds
+\int^1_{\ez} s^{-1}\phi\left(\frac s \ez\right)\,ds
=\int^1_{\ez} s^{-1}\phi\left(\frac s \ez\right)\,ds=v^\ez(\ez e_1).
\end{align}
 Via $\phi(s)=1$ when $s\ge 2$ and by \eqref{cons}, we get
$$\|v^{\ez}\|_{L^\fz(B_1)}\ge  v^{\ez}(|x|)=\int^1_{\ez} s^{-1}\phi\left(\frac s \ez\right)\,ds
\ge \int^1_{2\ez}s^{-1}\,ds=-\ln 2\ez.$$

On the other hand,  at any $x\in B_1\setminus B_\ez$, a direct calculation gives
\begin{align*}
&v^{\ez}_{x_i}=- \frac{x_i}{|x|^2}\phi(\ez^{-1}|x|),\\
&v^{\ez}_{x_ix_j}=-\ez^{-1} \frac{x_ix_j}{|x|^3}\phi'(\ez^{-1}|x|)
-\phi(\ez^{-1}|x|)[ \frac{\dz_{ij}}{|x|^2}-2  \frac{x_jx_i}{|x|^4}].
\end{align*}
Hence, at such $x$,
\begin{align*}
-|Dv^{\ez}|^{\gz}\bdz^N_{p} v^{\ez}(x)=\frac{1}{|x|^\gz }\phi^\gz(\ez^{-1}|x|)[(p-1)\frac 1{\ez|x|}\phi'(\ez^{-1}|x|)
+(n-2)\phi(\ez^{-1}|x|)\frac1{|x|^{ 2}}]=:g^\ez
\end{align*}
Therefore
\begin{align*}
\int_{B_1}(g^{\ez})^2\,dx&=c_n\int_\ez^1
\frac{1}{r^{2\gz} }\phi^{2\gz}(\ez^{-1}r)\left[(p-1)\frac 1{\ez r}\phi'(\ez^{-1}r)
+(n-2)\phi(\ez^{-1}r)\frac1{r^{ 2}}\right]^2\,r^{n-1}dr.
\end{align*}
Thus it follows from \eqref{hj-1} and \eqref{hj-2} that
\begin{align*}
&\int_{B_1}(g^{\ez})^2\,dx\\
&=c_n(p-1)\int_\ez^{2\ez}
\frac{1}{r^{2\gz} }\phi^{2\gz}(\ez^{-1}r)\phi'(\ez^{-1}r)\left[(p-1)\frac 1{(\ez r)^2}\phi'(\ez^{-1}r)
+2(n-2)\phi(\ez^{-1}r)\frac1{\ez r^{ 3}}\right]\,r^{n-1}dr\\
&\quad+c_n\int_{2\ez}^{1} (n-2)\frac{1}{r^{2\gz} } \phi^{2+2\gz}(\ez^{-1}r)\frac1{r^{ 4}}r^{n-
1}dr\\
&\le C(n,p)\ez^{n-4-2\gz}+ C(n)\int_{2\ez}^1r^{n-4-2\gz}\frac{dr}r.
\end{align*}
In the case of $\gz=\frac{n-4}2$, we have
$$
\int_{B_1}(g^{\ez})^2\,dx\le C\ln \frac1\ez. $$
Now let  $u^\ez =( \ln \ez^{-1})^{-1/2(1+\gz)}v^\ez$ when $ \gz=\frac{n-4}2$.
Thus
$$\int_{\Omega}(f^{\ez})^2\,dx\le C$$
while when $\gz=\frac{n-4}2$ and $n\ge 4$
 $$ \|u^\ez\|_{L^\fz(B_1)}\ge  C( \ln \frac1\ez )^{1/2(1+\gz)}  \|v^\ez\|_{L^\fz(B_1)}\ge C
\left( \ln\frac1\ez\right)^{\frac{n-3}{n-2}} \to \fz\ \mbox{as $\ez\to0$}.$$

If $\gz<\frac{n-4}2$, we have
$$
\int_{B_1}(g^{\ez})^2\,dx\le C  [\ez^{n-4-2\gz}+1]\le C. $$
Recall that
 $$ \|v^\ez\|_{L^\fz(B_1)}\to +\fz$$
as $\ez\to 0$. Hence we finish this proof.

\end{proof}

\section{Existence
of viscosity solutions under  $f\in L^\fz(\Omega)\cap C^0(\Omega)$ }
In this section, we always assume that $\Omega$ is convex domain, or
$\Omega$ is a bounded Lipschitz domain with
$\lim_{r\to 0^+}\Psi_{\Omega}(r)<\dz$ for some positive constant
$\dz=\dz(p,n,\gz, L_{\Omega},d_{\Omega})$ as in  Lemma \ref{Lq-es} (ii).
We also always let $p>1$ and let $\gz>-1$ satisfy \eqref{pgz}, let $k_0\ge 0$ be a any integer so that $(2^{\star})^{k_0+1}(\gz+1)>n$ and
 $\az=1-\frac{n}{(2^{\star})^{k_0+1}(\gz+1)}$.

\begin{lem}\label{ex-so}

Given any   $f \in C^0(  \Omega)\cap L^\fz(\Omega)$, there exists a viscosity solution $u\in C^{0,\az}(\overline \Omega)$ to the Dirichlet problem
       \begin{align}\label{app-gz}
\left\{
\begin{aligned}
-|Du|^{\gz}\bdz^N_p u &= f
&\quad{\rm in}\ \Omega,\\
u&=0&\quad {\rm on}\  \partial\Omega.
\end{aligned}
\right.
\end{align}
\end{lem}

\begin{rem}\rm

 Notice that $f$ may change sign and hence it is unclear for the uniqueness. Indeed, even if $f\ge 0$ or $f\le 0$ for
the normalized $p$-Laplacian $-\bdz^N_p u=f$( when $\gz=0$ ) the uniqueness  remains
open for $p\in (1,2)\cup (2,\fz)$; see  for example \cite{m23}.
\end{rem}

If $f\in C^0(\overline \Omega)$, the  existence of viscosity solution
$u\in C^0(\overline \Omega) $ to \eqref{app-gz} is a direct consequence of Birindelli and Demenge
\cite[Corollary 3.8]{bd09}.
 Their proof
also works when $f\in C^0(\Omega)\cap L^\fz(\Omega)$.
The contribution of Lemma \ref{ex-so} is to guarantee the
 existence of viscosity solutions with H\"older regularity.

To prove Lemma \ref{ex-so}, we use an approximate  argument.
Set $\bar f=f$ in $\Omega$ and $\bar f=0$ on
$\rr^n\backslash \Omega$. Now we set $f^{\ez}=\bar f\ast \eta^{\ez}$ where
$\eta^{\ez}$ is standard mollifier.
   For each $\ez\in (0,1]$ and $m\ge 1$,  let us consider the approximation problem:
\begin{align}\label{app-ex}
\left\{
\begin{aligned}
-(|Du^{\ez,m}|^2+\ez)^{\frac \gz 2}
\bdz^N_{p,\ez}u^{\ez,m}=f^\ez&\quad{\rm in}\ &\Omega_m,\\
u^{\ez,m}=0&\quad {\rm on}\ &\partial\Omega_m.
\end{aligned}
\right.
\end{align}
Here  $\Omega_m$  is smooth domain satisfying:
 \begin{enumerate}
\item[(i) ]$
\Omega\Subset  \Omega_{m+1} \Subset \Omega_m \quad\mbox{and}\quad \lim_{m\to\fz}\Omega_m=\Omega.$

\item[(ii) ]  $
L_{\Omega_m}\le  C(\Omega)L_{\Omega}\quad\mbox{and}\quad d_{\Omega_m} \le C(n) d_{\Omega}.
$

\item[(iii) ]
  $  \Psi_{\Omega_m}(r)\le C(\Omega) \Psi_{\Omega}  (r)$.
\end{enumerate}
Then we have the following:
\begin{lem}\label{ex-l2}
There exists a smooth solution $u^{\ez,m}\in C^\fz(\overline \Omega_m)$ to
\eqref{app-ex} such that
\begin{align}\label{un-hol}
\sup_{m\in\nn}\sup_{\ez\in(0,1]} \|u^{\ez,m}\|_{C^{0,\az}(\overline \Omega_m)}\le C_0[\|f\|_{L^\fz(\Omega)}^{\frac 1{\gz+1}}+1]
\end{align}
for some constant  $C_0:=C(\gz,p,n,d_{\Omega},L_{\Omega})$ and
for some $\az=\az(n,\gz)\in (0,1)$.
\end{lem}

\begin{proof}

By the standard quasilinear elliptic theory,
there exists a smooth solution $u^{\ez,m}\in C^\fz(\overline \Omega_m)$  provided that we have following a priori global estimates:
\begin{align}\label{es-m}
\|Du^{\ez}\|_{L^\fz(\Omega_m)}<+\fz.
\end{align}
Indeed, using \cite[Theorme 4.3]{cm14} by Cianchi-Maz'ya we get
$$\|(Du^{\ez}|^2+\ez)^{\frac12}\|_{L^\fz(\Omega_m)}
\le C(p,n,\Omega_m)[\|[|Du^{\ez}|^2+\ez]^{\frac{p-2-\gz}{2}}
|f^{\ez}|\|_{L^\fz(\Omega_m)}]^{\frac{1}{p-1}}.$$
If $p-2-\gz<0$, we then get
$$\|(Du^{\ez}|^2+\ez)^{\frac12}\|_{L^\fz(\Omega_m)}<C(p,n,\Omega_m)
\ez^{^{\frac{p-2-\gz}{2}}}\|f^{\ez}\|_{L^\fz(\Omega_m)}^{\frac 1{p-1}}<+\fz.$$
If $p-2-\gz\ge 0$, we also have
$$\|(Du^{\ez}|^2+\ez)^{\frac12}\|_{L^\fz(\Omega_m)}<C(p,n,\Omega_m)
\|\sqrt{|Du^{\ez}|^2+\ez}\|^{
\frac{p-2-\gz}{p-1}}_{L^\fz(\Omega_m)}\|f^{\ez}\|_{L^\fz(\Omega_m)}^{\frac 1{p-1}},$$
which leads to
$$\|(Du^{\ez}|^2+\ez)^{\frac12}\|^{\gz+1}_{L^\fz(\Omega_m)}<C(p,n,\Omega_m)
\|f^{\ez}\|_{L^\fz(\Omega_m)}<+\fz.$$
Thus \eqref{es-m} holds.

Next we show the uniform H\"older estimate.
Thanks to Lemmas \ref{holder-1} and \ref{holder-2}, by
the hypothesis (i)-(iii) on $\Omega_m$ we conclude that
\begin{align*}
 \|u^{\ez,m}\|_{C^{0,\az}(\overline \Omega_m)}\le C[\|f^{\ez}\|_{L^\fz(\Omega_m)}^{\frac 1{\gz+1}}+1]
+C\|u^{\ez,m}\|_{L^\fz(\Omega_m)},
\end{align*}
where $C=C(p,n,\gz,L_{\Omega},d_{\Omega})$ and $\az=\az(n,\gz)\in (0,1)$. Observe that
$$\|f^{\ez}\|_{L^\fz(\Omega_m)}\le \|\bar f\|_{L^\fz(\Omega_m)}
=\|f\|_{L^\fz(\Omega)}$$
due to $\bar f=0$ on $\Omega_m \backslash \Omega$ and $\bar f=f$
on $\Omega$. Thus
\begin{align}\label{hol}
 \|u^{\ez,m}\|_{C^{0,\az}(\overline \Omega_m)}\le C[\|f^{\ez}\|_{L^\fz(\Omega)}^{\frac 1{\gz+1}}+1]
+C\|u^{\ez,m}\|_{L^\fz(\Omega_m)}.
\end{align}

Now it remains to bound $L^\fz$-norm of $u^{\ez,m}$. We will use a comparison principle
in \cite[Proposition 2.5]{bd09} to prove this. Fix an  $y\in\partial\Omega_m$.
Assume that $ y=0$ without loss of generality.
Define $$v (x)= K[d_{\Omega_m}   -  |x| ],$$
 where $K\ge 1$ is given by
$$K=1+\left(\frac{2d_{\Omega}}{n-1}\|f\|_{L^\fz(\Omega)}\right)^{\frac 1{\gz+1}}.$$
We first show
\begin{align}\label{re-3}
u^{\ez,m}(x)\le v(x)\quad \forall x\in \Omega_m.
\end{align}

Since
$$Dv(x)=K|x|^{-1}x\quad
{\rm and}\quad D^2v(x)=K|x|^{-1}I_n-K|x|^{-3}x\otimes x,$$
then
$$\bdz_{\fz} v(x)=0\quad{\rm and}\quad
\bdz v(x)=K(n-1)|x|^{-1}.$$
Hence by  $|x|^{-1}\ge d_{\Omega_m}^{-1}$ we conclude that
\begin{align*}
-[|D v|^2+\ez]^{\frac\gz 2}
\bdz^N_{p,\ez} v& =K [ K^2  +\ez]^{\frac \gz 2}
( n-1)|x|^{-1}\\
&\ge d_{\Omega_m}^{-1}( n-1)K [ K^2  +\ez]^{\frac \gz 2}\\
&\ge d_{\Omega_m}^{-1}\frac{n-1}2K^{\gz+1}
\end{align*}
Here we also used the fact that
$$K [ K^2  +\ez]^{\frac \gz 2}\ge \frac 12K^{\gz+1}
$$
for all $K\ge 1$, $\ez\in (0,1)$ and $\gz>-1$. Recall that
$$K=1+\left(\frac{2d_{\Omega}}{n-1}\|f\|_{L^\fz(\Omega)}\right)^{\frac 1{\gz+1}}.$$
By $\|f^{\ez}\|_{L^\fz(\Omega_m)}\le \|f\|_{L^\fz(\Omega)}$ and
$d_{\Omega}\le d_{\Omega_m}$, one has
\begin{align*}
&-[|D v(x)|^2+\ez]^{\frac\gz 2}
\bdz^N_{p,\ez} v(x)> \|f^{\ez}\|_{L^\fz(\Omega_m)},
\quad \forall x\in \Omega_m.
\end{align*}
Since $u^{\ez}\le v$ on $\partial\Omega_m$, by a comparison principle
in \cite[Proposition 2.5]{bd09} one has \eqref{re-3}.

Now a similar argument also implies that
\begin{align*}
u^{\ez,m}(x)\ge -v(x),\quad \forall x\in \Omega_m,\quad \forall \ez\in (0,1),
\end{align*}
From this and \eqref{hol} we get \eqref{un-hol} as desired.

  \end{proof}

We finally prove Lemma \ref{ex-so}.
\begin{proof}
[Proof of Lemma \ref{ex-so}]
Thanks to  \eqref{un-hol},
we know that $ u^{\ez,m}\to  u^\ez $ in $C^{0,\az}(\overline \Omega)$ as $m\to\fz $(up to some subsequence)
for some $u^\ez\in C^{0,\az}(\overline\Omega)$ and $u^\ez=0$ on $\partial\Omega$. Moreover,
$u^{\ez}$ is a viscosity solution to
$$-(|Du^{\ez}|^2+\ez)^{\gz}\bdz^N_{p,\ez}u^{\ez}=f^{\ez}\quad{\rm in}\quad \Omega.$$
Note that $u^{\ez}\in C^{0,\az}(\overline \Omega)$ uniformly in $\ez\in (0,1)$ and
$f^{\ez}\to f$ in $C^0(\Omega)$ as $\ez\to 0$. Hence
by a stability result of viscosity solution (see \cite{ar18}) and
$u^{\ez}\to u$ in $C^{0,\az}(\overline \Omega)$ for some
$u\in C^{0,\az}(\overline \Omega)$ and $u=0$ on $\partial\Omega$, we conclude that $u$
is a viscosity solution to \eqref{app-gz}.

\end{proof}

\section{Proof of Theorem \ref{th-re0}}
This section is devoted to the proof of Theorems \ref{th-re0}. In this section, we always let $ (n,p,\gamma)$ be admissible.
We always assume that $\Omega$ is convex domain, or
$\Omega$ is a bounded Lipschitz domain with
$\lim_{r\to 0^+}\Psi_{\Omega}(r)<\dz$ for some positive constant
$\dz=\dz(p,n,\gz, L_{\Omega},d_{\Omega})$ as in  Lemma \ref{Lq-es} (ii).

 Let $u\in C^0(\Omega)$ be a viscosity solution to
\eqref{app-gz} with $f\in L^\infty(\Omega)\cap C^0(\Omega)$.
 We extend
$f$ to whole $\rn$ by setting $\bar f=f$ in $\Omega$ and
$\bar f=0$ on $\rr^n\backslash \Omega$. Similarly, we
set $\bar u=u$ on $\Omega$ and $\bar u=0$ on $\rr^n\backslash \Omega$.
Given $\ez\in (0,1)$, we define
$$g^{\ez}=\bar u\ast \eta_{\ez}\quad{\rm and}\quad
f^{\ez}=\bar f\ast \eta_{\ez},$$
where $\eta_{\ez}$ is standard mollifier.

We consider the following regularized equations.
 For each $\lz\in (0,1)$ and  $\epsilon\in(0,1)$,  consider problems
\begin{align}\label{app-un1}
\left\{
\begin{aligned}
-(|Dv|^2+\ez)^{\frac \gz 2}
\bdz^N_{p,\ez}v +\lz v&=\lz g^{\ez}+f^\ez
&\quad{\rm in}\ &\Omega\\
v&=0&\quad {\rm on}\ &\partial\Omega.
\end{aligned}
\right.
\end{align}Let $u^{\ez}\in C^0(\overline \Omega)$ be the unique viscosity solution to
\eqref{app-un1}; see for example \cite{ar18}.
Moreover, denote by  $\{\Omega_m\}_{m\ge 1}$ satisfies (i)-(iii) as in
Section 6.
For each $m\in\nn$, consider the problem
\begin{align}\label{app-un}
\left\{
\begin{aligned}
-(|Dv|^2+\ez)^{\frac \gz 2}
\bdz^N_{p,\ez}v +\lz v&=\lz g^{\ez}+f^\ez
&\quad{\rm in}\ &\Omega_m\\
v&=0&\quad {\rm on}\ &\partial\Omega_m.
\end{aligned}
\right.
\end{align}

 \begin{lem}\label{re-es}
Given each $\ez\in(0,1)$ and $m\in\nn$, there exists a unique solution
$u^{\ez,m}\in C^\fz(\overline \Omega_m)$ to \eqref{app-un} such that
\begin{align}\label{app-hol}
 \sup_{m\in\nn}\sup_{\ez\in(0,1]}\|u^{\ez,m}\|_{C^{0,\alpha}(\overline\Omega_m )}\le C[\||f |+\lz|u|\|^{\frac{1}{\gz+1}}_{L^\fz(\Omega)}+1]
\end{align}
where $C=C(p,n,\gz,d_{\Omega},L_{\Omega})$ and
$\az=1-\frac{n}{(2^{\star})^{k_0+1}(\gz+1)}$ for $k_0>\ln \frac{n}{2^{\star}(\gz+1)}$.
\end{lem}

\begin{proof}
 As in Lemma \ref{ex-l2}, the existence of smooth solution to \eqref{app-un} follows from a priori estimate \eqref{app-hol}. Also, the uniqueness is a direct consequence of  a comparison principle  in \cite[Proposition 2.5]{bd09}. Hence, we only need to prove \eqref{app-hol}.

We first establish the following a priori $L^\fz$-estimate:
\begin{align}\label{app-lf}
\|u^{\ez,m}\|_{L^\fz(\Omega_m)}\le C[\||f |+\lz|u|\|^{\frac{1}{\gz+1}}_{L^\fz(\Omega)}+1],
\end{align}
where $C=C(p,n,\gz,d_{\Omega},L_{\Omega})$. One may assume that
$0\in \partial\Omega_m$. Define
$$v(x)=K(d_{\Omega_m}-|x|)$$
where
$$K=1+\left(\frac{2d_{\Omega}}{n-1}\||f|+\lz|u|\|_{L^\fz(\Omega)}\right)^{\frac 1{\gz+1}}$$
By a calculation as in Lemma \ref{ex-l2} and $\lz v\ge 0$ on $\Omega_m$, we have
$$-(|Dv|^2+\ez)^{\frac \gz 2}\bdz^N_{p,\ez}v+\lz v\ge d_{\Omega_m}^{-1}\frac{n-1}2K^{\gz+1}
$$
Now by our choice of $K$, using a comparison principle
in \cite[Proposition 2.5]{bd09} one has
$$u^{\ez,m}\le v\quad{\rm in}\quad \Omega_m.$$
Similarly, we also have $u^{\ez,m}\ge -v$ on
$\Omega_m$. This proves \eqref{app-lf}.

Write
 $$G_\ez= \lz (g^\ez-u^{\ez,m})+f^\ez.$$
Thanks to \eqref{app-lf}, by an argument of in Lemma \ref{ex-l2} we conclude
\eqref{app-hol} as desired.
\end{proof}
By a stability of viscosity solutions, we have:
\begin{lem}\label{conver}
Let $\az\in (0,1)$ as in Lemma \ref{re-es}.
It holds that
   \begin{align}\label{un-con1}
\mbox{ $u^{\ez,m}\to u^\ez$ in   $C^{0,\alpha}(\overline \Omega)$
as $m\to \fz$ and $Du^{\ez,m}\to Du^{\ez}$ in
$C^0_{\loc}(\Omega)$ as $m\to \fz$.}
\end{align}
Moreover, we have
   \begin{align}\label{un-con2}
\mbox{ $u^{\ez}\to u$ in   $C^{0,\alpha}(\overline \Omega)$
as $\ez\to 0$.}
\end{align}
\end{lem}

\begin{rem}
It is not clear whether $Du^{\ez}\to Du$ in $C^{0,\az}_{\loc}(\Omega)$ as
$\ez\to 0$.
\end{rem}

 \begin{proof}
 Thanks to \eqref{app-hol} and $u^{\ez,m}=0$ on $\partial\Omega_m$, via a stability of viscosity solution one has
$u^{\ez,m}\to v^{\ez}$ in $C^{0,\az}(\overline \Omega)$ as $m\to \fz$ for
some $v^{\ez}\in C^0(\overline \Omega)$.
Observe that $v^{\ez}=0$ on $\partial\Omega$. By the uniqueness, we have $u^\ez=v^\ez$.
By \eqref{app-hol} again and \cite[Theorem 1.1]{ar18}, we know that $u^{\ez,m}\in C^1_{\loc}(\Omega)$ uniformly
in $m$ and hence $Du^{\ez,m}\to Du^{\ez}$ in $C^0_{\loc}(\Omega)$ as $m\to \fz$.
This proves \eqref{un-con1}. Similarly, by \eqref{app-hol} again
there exists $v\in C^0(\overline \Omega)$ so that $u^{\ez}\to v$
in $C^{0,\az}(\overline \Omega)$. Moreover, $v$ is a viscosity solution of
$$-|Dv|^{\gz}\bdz^N_p v+\lz v=\lz u+f\quad{\rm in}\quad
\Omega; \ v=0\quad{\rm on}\quad \partial\Omega$$
according to $g^{\ez}\to u$ in $C^0(\Omega)$ and $f^{\ez}\to f$ in $C^0(\Omega)$ as
$\ez\to 0$. So the uniqueness also implies that $v=u$.
Hence we complete this proof.

 \end{proof}
We also need to the following global convergence property.
\begin{lem}\label{con-gr}

(i) 
 Given any $\ez\in(0,1)$, it holds that
\begin{align}\label{gr-co1}
\mbox{ $(|Du^{\ez,m}|^2+\ez)^{\frac \gz 2}Du^{\ez,m}\to (|Du^{\ez}|^2+\ez)^{\frac \gz 2}Du^{\ez}$ in   $L^2(\Omega)$
as $m\to \fz$}
\end{align}
(ii) 
It holds that 
 $Du^{\ez}\to
Du$ in $L^q(\Omega)
$ for all $q<(1+\gamma)2^\star$ and that
\begin{align}\label{gr-co2}
\mbox{ $(|Du^{\ez}|^2+\ez)^{\frac \gz 2}Du^{\ez}\to |Du|^{\gz}Du$ in   $L^q(\Omega)$
for all $q<2^\star$ as $\ez\to 0$.}
\end{align}

\end{lem}

\begin{proof}
{\it Proof of (i)}.
 Define
$$G_{\ez,m}=\lz (g^{\ez}-u^{\ez,m})+f^{\ez}.$$
It follows by \eqref{app-hol} that
\begin{align}\label{un-lfz}
\sup_{\ez\in (0,1),m>1}\|G_{\ez,m}\|_{L^\fz(\Omega_m)}<+\fz.
\end{align}
Applying $u^{\ez,m}$ to $v$ and $\Omega_m$ to $\Omega$ in Lemma \ref{Lq-es}, with
the conditions (i)-(iii) of $\Omega_m$ as in Sections 5
one concludes that
\begin{align}\label{gr-eq1}
\left(\int_{\Omega_m}(|Du^{\ez,m}|^2+\ez)^{\frac{2^{\star}}2(\gz+1)}\,dx\right)^{\frac 1{2^{\star}(\gz+1)}}
&\le C\left(\int_{\Omega_m} G_{\ez,m}^2\,dx\right)^{\frac 1{2(\gz+1)}}
+C\ez^{\frac 12},
\end{align}
where $C=C(p,n,\gz,L_{\Omega},d_{\Omega})$. From this and \eqref{un-lfz} we know that
$(|Du^{\ez,m}|^2+\ez)^{\frac \gz 2}Du^{\ez,m}\in L^{2^{\star}}(\Omega)$ uniformly
in $m>1$ and $\ez\in (0,1)$.  For each smooth domain $\Omega'\Subset \Omega$, via H\"older inequality we get
\begin{align*}
&\int_{\Omega}\left|(|Du^{\ez,m}|^2+\ez)^{\frac \gz 2}Du^{\ez,m}-(|Du^{\ez}|^2+\ez)^{\frac \gz 2}Du^{\ez}\right|^2\,dx\\
&=\int_{\Omega\backslash\Omega'}\left|(|Du^{\ez,m}|^2+\ez)^{\frac \gz 2}Du^{\ez,m}-(|Du^{\ez}|^2+\ez)^{\frac \gz 2}Du^{\ez}\right|^2\,dx\\
&\quad
+\int_{\Omega'}\left|(|Du^{\ez,m}|^2+\ez)^{\frac \gz 2}Du^{\ez,m}-(|Du^{\ez}|^2+\ez)^{\frac \gz 2}Du^{\ez}\right|^2\,dx\\
&\le |\Omega\backslash\Omega'|^{\frac{2^{\star}-2}{2^{\star}}}
\left(\int_{\Omega\backslash\Omega'}\left|(|Du^{\ez,m}|^2+\ez)^{\frac \gz 2}Du^{\ez,m}-(|Du^{\ez}|^2+\ez)^{\frac \gz 2}Du^{\ez}\right|^{2^{\star}}\,dx
\right)^{\frac{2}{2^{\star}}}\\
&\quad
+\int_{\Omega'}\left|(|Du^{\ez,m}|^2+\ez)^{\frac \gz 2}Du^{\ez,m}-(|Du^{\ez}|^2+\ez)^{\frac \gz 2}Du^{\ez}\right|^2\,dx.
\end{align*}
By \eqref{un-con1}, then the second term of right-hand side tends to zero as
$m\to \fz$. Then \eqref{gr-co1} follows by letting $\Omega'\to \Omega$ and
$(|Du^{\ez,m}|^2+\ez)^{\frac \gz 2}Du^{\ez,m}\in L^{2^{\star}}(\Omega)$ uniformly
in $m>1$.

{\it Proof of (ii)}.
 Set
$$q_0=2^{\star}(\gz+1).$$
When $n=2$, we can always assume that  $2^\star>\frac2{1+\gamma}$ and hence $ q_0>2$.
 When $n\ge3$, since  $\gz>-\frac{4}{n+2}$ and  $2^\star-1=\frac{n+2}{n-2}$,
 we have
\begin{align}\label{ind-2}
 q_0=2^{\star}( 1+\gamma) = \gamma +2^{\star}+(2^{\star}-1)\gamma >  \gamma +
 \frac{2n}{n-2}-\frac{n+2}{n-2} \frac{4}{n+2}=\gamma+2,
\end{align}
in particular, $ q_0>1$ and
$2^\star>\frac{q_0}{q_0-1}$.

Write $q_0^\star=\frac{nq_0}{n-q_0}$ if $ q_0<n$ and $q_0^\star=q$ for some sufficiently large $q\in(2,\infty)$ if $q_0\ge n$.
Observe  that  $q_0^\star>2$; indeed, if   $q_0< n$, by $q_0>\gamma+2$ we have
 \begin{align}\label{ind-1}
 \frac{nq_0}{n-q_0}> \frac{n(2-\frac4{n+2}) }{n-(2-\frac4{n+2})}=\frac{n(2n+4-4)}{n(n+2)-(2n+4-4)}
 =2\quad{\rm if }\quad q_0<n.
 \end{align}

Recall that  Lemma \ref{Lq-es} tells us that
 $D u^{\ez}\in L^{q_0}(\Omega)$
 uniformly in
$\ez\in (0,1)$.
Thanks to \eqref{ind-1} and \eqref{un-con2}, we conclude that
\begin{align}\label{xx2}
u^{\ez}\to u \quad{\rm in}\quad L^{2}(\Omega)
\quad{\rm and}\quad Du^{\ez}\to
Du\quad{\rm weakly\ in}\quad L^{q_0}(\Omega)
\end{align}
as $\ez\to0 $ (up to some subsequence).
Moreover,
we claim that
\begin{align}\label{xx3}
&\frac 1{C}\int_{\Omega}(|Du^{\ez}|^2+|Du|^2+\ez)^{\frac \gz 2}|Du^{\ez}-Du|^{2}\,dx \to 0
\end{align}
as $\ez\to0$.

To see the claim \eqref{xx3}, by \eqref{in-vec1} and \eqref{in-vec3} in Lemma \ref{in-vec}, one has
\begin{align}\label{feqq}
&(|Du^{\ez}|^2+|Du|^2+\ez)^{\frac \gz 2}|Du^{\ez}-Du|^{2}\nonumber\\
&\le C(\gz)\left((|Du|^2+\ez)^{\frac \gz2}Du-(|Du^{\ez}|^2+\ez)^{\frac \gz2}Du^{\ez}\right)
\cdot (Du-Du^{\ez}).
\end{align}
The fact $\gz+2<q_0$ allows us to integrate over $\Omega$ and then obtain
\begin{align}\label{xx1}
&\frac 1{C}\int_{\Omega}(|Du^{\ez}|^2+|Du|^2+\ez)^{\frac \gz 2}|Du^{\ez}-Du|^{2}\,dx\nonumber\\
&\le
\left(
\int_{\Omega}(|Du|^2+\ez)^{\frac \gz2}|Du|^{2}\,dx-\int_{\Omega}(|Du|^2+\ez)^{\frac \gz2}Du\cdot Du^{\ez}\,dx\right)\nonumber\\
&\quad
+\int_{\Omega}(|Du^{\ez}|^2+\ez)^{\frac \gz2}Du^{\ez}\cdot (Du^{\ez}-Du)\,dx.
\end{align}

Observe that $Du\in L^{q_0}(\Omega)$ implies   $|Du|^{\gz+1}\in L^{2^\star}(\Omega)$.
Since  $\gz>-\frac{4}{n+2}$ implies $2^\star>\frac{q_0}{q_0-1}$, we have
 $|Du|^{\gz+1}\in L^{\frac{q_0}{q_0-1}}(\Omega)$.
  Therefore, by
$Du^{\ez}\to Du$ weakly in $L^{q_0}(\Omega)$ and
$(|Du|^2+\ez)^{\frac \gz2}Du\to |Du|^{\gz}Du$ uniformly convergence
one has that
$$\int_{\Omega}(|Du|^2+\ez)^{\frac \gz2}Du\cdot Du^{\ez}\,dx\to \int_{\Omega}|Du|^{\gz+2}\,dx.$$
Then the first term in the right-hand sides of \eqref{xx1} tends to zero as
$\ez\to 0$.

It remains to show that the second term in the right-hand sides of \eqref{xx1} converges to $0$ as $ \ez\to0$.
Recall that $(|Du^{\ez}|^2+\ez)^{\frac \gz2}Du^{\ez}\in W^{1,2}(\Omega)\cap L^{2^{\star}}(\Omega)
$ uniformly in $\ez\in (0,1)$. There exists
$U\in W^{1,2}(\Omega)\cap L^{2^{\star}}(\Omega)$ such that
\begin{align}\label{str-con}
(|Du^{\ez}|^2+\ez)^{\frac \gz2}Du^{\ez}\to U\quad{\rm in}\quad L^{q}(\Omega)\quad
{\rm for\ all}\quad q<2^{\star}\ \mbox{and weakly in }\ W^{1,2}(\Omega).
\end{align}
as $ \ez\to0$ (up to some subsequence).
By \eqref{str-con} with $q=\frac{q_0}{q_0-1}$,
 $\frac{q_0}{q_0-1}<2^{\star}$ as given by \eqref{ind-2}, and \eqref{xx2}, we conculde that
\begin{align*}
&\int_{\Omega}(|Du^{\ez}|^2+\ez)^{\frac \gz2}Du^{\ez}\cdot (Du^{\ez}-Du)\,dx\to0
\end{align*}
as $\ez\to0$  (up to some subsequence) as desired. 

Next we show that \begin{align}\label{xx4}
\mbox{$Du^{\ez}\to Du$ in $L^{2+\gz}(\Omega)$ as $\ez\to0$ (up to some subsequence). }
\end{align}
If $\gz>0$,  \eqref{xx4} follows from  \eqref{xx3} and
$$(|Du^{\ez}|^2+|Du|^2+\ez)^{\frac \gz 2}|Du^{\ez}-Du|^{2}\ge C(\gz)
|Du^{\ez}-Du|^{\gz+2}.$$
 If $\gz<0$,   via H\"older inequality
\begin{align*}
&\int_{\Omega}|Du^{\ez}-Du|^{\gz+2}\,dx\\
&\le \left(\int_{\Omega}(|Du^{\ez}|^2+|Du|^2+\ez)^{\frac \gz 2}|Du^{\ez}-Du|^{2}\,dx\right)^{\frac{\gz+2}2}
\left(\int_{\Omega} (|Du^{\ez}|^2+|Du|^2+\ez)^{\frac{\gz+2}2}\,dx\right)^{-\frac \gz 2}.
\end{align*}
From this, then \eqref{xx4} follows  from  \eqref{xx3},
and  $Du,\, Du^\ez\in L^{\gamma+2}(\Omega)$ uniformly in $ \ez>0$.

By \eqref{str-con} and \eqref{xx4}, one has  $ U=|Du|^\gamma Du$ as desired.
This further implies that $Du^\ez\to Du$ in $L^q(\Omega)$
 as $\ez\to0$ (up to some subsequence) for all $q<q_0$.
\end{proof}

\begin{rem}\rm \label{zzzz1}
In the case $n=2$,  since $\frac4{n+2}=1$,
$ \gamma>-1$ then implies
 $\gamma>-\frac4{n+2}$ needed  in Lemma  \ref{con-gr} (ii) .

 In the case  $n\ge3$,
 the addition restriction $\gamma>-\frac4{n+2}$  in Lemma  \ref{con-gr} (ii)
 comes from our approach, where we require $2^\star(1+\gamma)>\gamma+2$.
 A  possible way to remove the restriction $\gamma>-\frac4{n+2}$ is to prove
  $Du^{\ez}\to Du$ in $C^{0,\az}_{\loc}(\Omega)$ as $\ez\to 0$, which will guarantee that
  \eqref{gr-co2}.
  However, this needs much more efforts. We leave this for further study.

\end{rem}

Now we are ready to prove Theorem \ref{th-re0}.
\begin{proof}[Proof of Theorem \ref{th-re0}]

The following cases are delineated.

 (i) $\Omega$ is convex. In this case, we also have that $ \Omega_m$ is convex such
that $\lim_{m\to \fz}|\Omega_m \backslash \Omega|=0$. Then \eqref{th1-1} in Theorem \ref{thm-v} give us that
\begin{align*}
\|(|Du^{\ez,m}|^2+\ez)^{\frac \gz 2}Du^{\ez,m}\|^2_{W^{1,2}(\Omega)}\le
C\int_{\Omega_m}(f^{\ez}+\lz (u^{\ez,m}-g^{\ez}))^2\,dx+C\ez^{\gz+1},
\end{align*}
where $C=C(p,n,\gz)$.
Set $$G_{\ez,m}=\lz (g^{\ez}-u^{\ez,m})+f^{\ez}.$$
Then by  \eqref{un-lfz} we have that $(|Du^{\ez,m}|^2+\ez)^{\frac \gz 2}Du^{\ez,m}\in
W^{1,2}(\Omega)$ uniformly in $m>1$. Hence using a weak compactness of Sobolev space
and \eqref{gr-co1} we conclude that
\begin{align*}
\|(|Du^{\ez}|^2+\ez)^{\frac \gz 2}Du^{\ez}\|^2_{W^{1,2}(\Omega)}&\le \limsup_{m\to \fz}
\|(|Du^{\ez,m}|^2+\ez)^{\frac \gz 2}Du^{\ez,m}\|^2_{W^{1,2}(\Omega)}\\
&\le
C\limsup_{m\to \fz}\int_{\Omega_m}(G_{\ez,m})^2\,dx+C\ez^{\gz+1}.
\end{align*}
Note that $\|G_{\ez,m}\|_{L^\fz(\Omega_m)}<\fz$ uniformly in $m$ and
$|\Omega_m\backslash \Omega|\to 0$ as $m\to \fz$. Hence
by $u^{\ez,m}\to u^{\ez}$ in $C^0(\overline \Omega)$ one has
\begin{align*}
\limsup_{m\to \fz}\int_{\Omega_m}(G_{\ez,m})^2\,dx&=\limsup_{m\to \fz}\int_{\Omega_m\backslash \Omega}(G_{\ez,m})^2\,dx+\limsup_{m\to \fz}\int_{\Omega}(G_{\ez,m})^2\,dx\\
&\quad =
\int_{\Omega}(f^{\ez}+\lz (u^{\ez}-g^{\ez}))^2\,dx,
\end{align*}
which leads to
\begin{align*}
\|(|Du^{\ez}|^2+\ez)^{\frac \gz 2}Du^{\ez}\|^2_{W^{1,2}(\Omega)}\le
C\int_{\Omega}(f^{\ez}+\lz (u^{\ez}-g^{\ez}))^2\,dx+C\ez^{\gz+1}.
\end{align*}
From this, it follows by \eqref{un-lfz} again and \eqref{gr-co2} that
\begin{align*}
\||Du|^{\gz}Du\|^2_{W^{1,2}(\Omega)}&\le \limsup_{\ez\to 0}
\|(|Du^{\ez}|^2+\ez)^{\frac \gz 2}Du^{\ez}\|^2_{W^{1,2}(\Omega)}\\
&\le
C\limsup_{\ez\to 0}\int_{\Omega}(f^{\ez}+\lz (u^{\ez}-g^{\ez}))^2\,dx\\
&=C\int_{\Omega}f^2\,dx,
\end{align*}
where we used dominated convergence theorem in the last inequality.

(ii) {\it $\Omega$ is not convex}.
Observe that $f\in L^\fz(\Omega)\cap C^0(\Omega)$. Then all estimates
uniformly in $\ez$ and $m$ and hence this  proof of this is similar to
the case (i)  by using \eqref{th1-2} in Theorem \ref{thm-v} and
 Lemma \ref{con-gr}.

\end{proof}

\begin{rem}\rm \label{rem7.6}
The  viscosity solution  $u$ obtained in Theorem \ref{th-re0} have the following quantitative (local) H\"older regularity.

\begin{enumerate}
\item[(i)] {\it Case  $\Omega$ is convex.}

  If $ (1+\gamma)2^\star>n$, then  $u\in C^{1-\frac n{(1+\gamma)2^\ast}}(\overline \Omega)$ with
$$\|u\|_{C^{1-\frac n{(1+\gamma)2^\ast}}(\overline \Omega)}\le C
\|Du\|_{L^{(1+\gamma)2^\star}(\Omega)}\le  C\|f\|^{\frac1{1+\gamma}}_{L^2(\Omega)}.$$

   If $ (1+\gamma)2^\star\le n$, then
$$\|u\|_{C^{\alpha}(\overline{B_{r/2}(y)} )}\le C \|f\|_{L^\infty (B_{r}(y))}  +C \|f\|_{L^2(\Omega)}\quad\mbox{whenever  $B_{2r}(y)\subset\Omega$} $$
for some $\alpha>0$.
Indeed,  applying Lemma  \ref{local-hol}  to $u^{\ez,m}$ we have
\begin{align}\label{holdereps}
\|u^{\epsilon,m} \|_{C^{0,\az}(B_{r/2}(y))}\le C\|G_{\ez,m}\|
^{1/(\gz+1)}_{L^\fz(B_r(y))}+C\|G_{\ez,m}\|^{\frac 1{\gz+1}}_{L^2(\Omega)}+C\ez^{\frac 12}.
\end{align}
Letting $m\to\infty$ and $\ez\to0$ in order we get the desired result.

\item[(ii)] {\it Case  $\Omega$ is not convex.}

If $\gamma\le p-2$, we have similar results as the convex case.

  If $ \gamma> p-2$ and $ (1+\gamma)2^\star>n$, then  $u\in C^0(\overline \Omega)$
$$\|u\|_{C^{1-\frac n{(1+\gamma)2^\ast}}(\overline \Omega)}\le C
\|Du\|_{L^{(1+\gamma)2^\star}(\Omega)}\le C \|f\|^{1/(1+\gamma)}_{L^2(\Omega)}+C\|u\|^{1/(1+\gamma)}_{L^2(\Omega)}.$$

  If   $ \gamma> p-2$ and $ (1+\gamma)2^\star\le n$, then by a similar argument we also have
$$\|u\|_{C^{\alpha}(\overline{B_{r/2}(y)} )}\le C \|f\|_{L^\infty (B_{r}(y))}  +C \|f\|_{L^2(\Omega)}+C \|u\|^{1/(1+\gamma)}_{L^2(\Omega)}
\quad\mbox{whenever  $B_{2r}(y)\subset\Omega$} $$
for some $\alpha>0$.  The proof is much similar to the convex case.
\end{enumerate}
\end{rem}

\section{Proofs of  Theorem \ref{th-re1} and Theorem \ref{th-re1a}}

We only prove Theorem  \ref{th-re1}; the  proof of Theorem  \ref{th-re1a} is much similar.
Assume that $\Omega$ is a convex and  $f\in L^2(\Omega)$. Let $f^\epsilon\in  C^0(\overline \Omega)
$ such that $ f^\ez\to f$ in $L^2(\Omega)$ as $\ez\to 0$. For $\epsilon>0$, denote by $u^\epsilon\in C^0(\overline \Omega)$ a  viscosity solution to the  approximation problem
\begin{equation}\label{ez-app}
-|Dv|^\gamma \Delta^N_p v=f^\epsilon \ \mbox{in $\Omega$};v=0 \ \mbox{on $\partial\Omega$.}
\end{equation}
  Applying   Theorem \ref{th-re0}, we know that
  \begin{align}\label{th1-1eps0}Du^\ez\in L^{(1+\gamma)2^\star}(\Omega) \ \mbox{   and}\
  |Du^\epsilon|^\gamma Du^\epsilon\in W^{1,2}(\Omega)\end{align} with
  \begin{align}\label{th1-1eps}
\|Du^\ez\|^{\gz+1}_{L^{(1+\gz)2^\star}(\Omega)}+
\|D[|Du^\ez|^{\gz}Du^\ez]\|_{L^2(\Omega)}\le C\|f ^\ez\|_{L^2(\Omega)}\le C\|f\|_{L^2(\Omega)},
\end{align}
where $C=C(p,n,\gz)$.
 Here and below set $q_0=(1+\gamma)2^\star$.
We then conclude that there exists a function $u\in W^{1,q_0}_0(\Omega)$ such that
\begin{align}\label{xx2-1}
u^{\ez}\to u \quad{\rm in}\quad L^{2}(\Omega)
\quad{\rm and}\quad Du^{\ez}\to
Du\quad{\rm weakly\ in}\quad L^{q_0}(\Omega)
\end{align}
as $\ez\to0 $ (up to some subsequence). Moreover,
we have the following convergence result, which can be proved in a similar way
by an approach in Lemma \ref{con-gr} (ii).
\begin{lem}\label{conv-res-}

It holds that
 $Du^{\ez}\to
Du$ in $L^q(\Omega)
$ for all $q<(1+\gamma)2^\star$ and that
 \begin{align}\label{str-con-v}
|Du^{\ez}|^{\gz} Du^{\ez}\to |Du|^{\gz}Du\quad{\rm in}\quad L^{q}(\Omega)\quad
{\rm for\ all}\quad q<2^{\star}.
\end{align}
\end{lem}
\begin{proof}
In the case $\gz>0$, by replacing \eqref{in-vec3} with \eqref{in-vec4},
the proof of Lemma 8.1 follows from that of Lemma \ref{con-gr} (ii) line by line.

In the case $\gz<0$, by replacing \eqref{in-vec1} with \eqref{in-vec2},
 the proof of Lemma 8.1 is also similar to that of Lemma \ref{con-gr} (ii).
 We only give the difference.
 If $-1<\gz<0$, it follows from \eqref{in-vec2} that
\begin{align*}
&\frac 1{C}\int_{\Omega}(|Du^{\ez}|^2+|Du|^2+1)^{\frac \gz 2}|Du^{\ez}-Du|^{2}\,dx\\
&\le \int_\Omega (|Du^\ez|^\gz Du^\ez-|Du |^\gz Du) (Du^\ez-Du)\,dx\\
&= \left(
\int_{\Omega}|Du|^{2+\gz}\,dx-\int_{\Omega}|Du|^{\gz}Du\cdot Du^{\ez}\,dx\right)+\int_{\Omega}|Du^{\ez}|^{\gz}Du^{\ez}\cdot (Du^{\ez}-Du)\,dx.
\end{align*}
By an argument similar to that of Lemma \ref{con-gr} (ii), one has
$$\int_{\Omega}(|Du^{\ez}|^2+|Du|^2+1)^{\frac \gz 2}|Du^{\ez}-Du|^{2}\,dx\to 0$$
as $\ez\to 0$ (up to some subsequence). Since a H\"older inequality leads to
\begin{align*}
&\int_{\Omega}|Du^{\ez}-Du|^{\gz+2}\,dx\\
&\le \left(\int_{\Omega}(|Du^{\ez}|^2+|Du|^2+1)^{\frac \gz 2}|Du^{\ez}-Du|^{2}\,dx\right)^{\frac{\gz+2}2}
\left(\int_{\Omega} (|Du^{\ez}|^2+|Du|^2+1)^{\frac{\gz+2}2}\,dx\right)^{-\frac \gz 2},
\end{align*}
we know that  $ Du^\ez\to Du$ in $L^{2+\gamma}(\Omega)$.
By an argument similar to   that of Lemma \ref{con-gr} (ii), we get the desired convergence.
 \end{proof}

We are ready to prove  Theorem  \ref{th-re1}.
\begin{proof}[Proof of Theorem  \ref{th-re1}]

Thanks to Lemma \ref{conv-res-} and \eqref{th1-1eps}, we conclude that
 $|Du^\ez|^{\gz}Du^\ez\to |Du|^{\gz}Du$ in $L^2(\Omega)$ and $D(|Du^\ez|^{\gz}Du^\ez)\to D(|Du|^{\gz}Du)$
 weakly in $L^2(\Omega)$ as $\ez\to 0$.
 By such convergence we conclude \eqref{thxx} and \eqref{th1-1} follows from
   \eqref{th1-1eps0} and \eqref{th1-1eps}.

     Finally, assume that $f\in   f^0(\Omega)$ in addition.

  \medskip \noindent
{\it Case    $(1+\gamma)2^\star>n$.}  By the Morrey-Sobolev embedding,
  $u^\epsilon\in C^{\alpha}(\overline \Omega)$ uniformly in $\ez>0$, where
  $\alpha=1-\frac n{(1+\gamma)2^\star}$.
By Arezla-Ascolli' theorem, as $\ez\to0$ (up to some subsequence), $u^\epsilon  $ uniformly converges to some function $u$ in $ C^{\alpha}(\overline\Omega)$. By the stability of viscosity solution,  we know that $u$
is a viscosity solution to \eqref{plapgz}.

\medskip
\noindent {\it Case $\gamma<\frac{n-4}2$}.
   Due to Remark \ref{rem7.6},  we have
 $u^\epsilon\in C^\alpha_{\loc}( \Omega)$ uniformly in $\ez>0$ for some $ \alpha>0$.   Thus
 $u^\ez\to u$ local uniformly as $\ez\to0$. Thus $ u\in C^0(\Omega)$.
\end{proof}

Below we explain  the  reason why Theorem \ref{th-re1a} does not include the case $\gz>p-2$.
\begin{rem} \rm \label{rem7.4}
Assume that   $\gz>p-2$.
In an domain $ \Omega$ as in Theorem \ref{th-re1a},  applying \eqref{th1-2}  in Theorem \ref{th-re0} to viscosity solution $u^\epsilon$ to approximation problem,  we will get
\begin{align*}
\||Du^{\ez}|^\gamma Du^{\ez}\|_{W^{1,2}(\Omega)}\le C\|f^{\ez}\|_{L^2(\Omega)}+C\|u^{\ez}\|
^{\gz+1}_{L^2(\Omega)}.
\end{align*}
Since it is not clear whether  $\|u^\epsilon\|_{L^2(\Omega)}$  is bounded uniformly in $\ez>0$ or not,   we can not use the above proof to show \eqref{th1-2} for $u$.
\end{rem}
\section*{Data availability statement }
Data sharing not applicable to this article as no datasets were generated or analysed
during the current study.

\section*{Statements and Declarations}

On behalf of all authors, the corresponding author states that there is no conflict of interest.

\noindent Qianyun Miao,

\noindent
School of Mathematics and Statistics, Beijing Institute of Technology, Beijing 100081, P. R. China.

\noindent{\it E-mail }:  \texttt{qianyunm@bit.edu.cn}
\bigskip

\noindent Fa Peng,

\noindent
School of Mathematical Sciences, Beihang University, Beijing 102206, P. R. China

\noindent{\it E-mail }:  \texttt{fapeng@buaa.edu.cn}

\bigskip

\noindent Yuan Zhou

\noindent School of Mathematical Sciences, Beijing Normal University, Haidian District Xinejikou Waidajie No.19, Beijing 100875, P. R. China

\noindent {\it E-mail }:
\texttt{yuan.zhou@bnu.edu.cn}

\end{document}